\documentclass[10pt,a4paper]{article}

\usepackage{amsmath}   %\mathcal
\usepackage{amssymb}   %basic
\usepackage{amsthm}    %theoremstyle
\usepackage{mathrsfs}  %\mathscr
\usepackage{enumerate}
\usepackage{parskip}
\usepackage{fullpage}

\theoremstyle{plain}
\newtheorem{theorem}{Theorem}[section]
\newtheorem{corollary}[theorem]{Corollary}
\newtheorem{lemma}[theorem]{Lemma}
\newtheorem{proposition}[theorem]{Proposition}
\theoremstyle{definition}
\newtheorem{definition}[theorem]{Definition}
\theoremstyle{remark}
\newtheorem{remark}[theorem]{Remark}

\makeatletter
\def\thm@space@setup{%
  \thm@preskip=\parskip \thm@postskip=0pt
}
\makeatother

\DeclareMathOperator*{\esssup}{ess\,sup}
\DeclareMathOperator*{\essinf}{ess\,inf}
\DeclareMathOperator{\tr}{tr}
\newcommand{\cO}{\mathcal{O}}
\newcommand{\cL}{\mathcal{L}}
\newcommand{\cM}{\mathcal{M}}
\newcommand{\cS}{\mathcal{S}}
\newcommand{\bP}{\mathbb{P}}
\newcommand{\bR}{\mathbb{R}}
\newcommand{\bN}{\mathbb{N}}
\newcommand{\sF}{\mathscr{F}}
\newcommand{\sP}{\mathscr{P}}

\makeatletter\@addtoreset{equation}{section} \makeatother
\begin{document}

\title{{A Non-Markovian Liquidation Problem and Backward SPDEs with Singular Terminal Conditions}\thanks{We thank seminar participants at various institution for valuable comments and suggestion. Financial support from the CRC 649 \textit{Economic Risk} and \textit{d-fine} GmbH is gratefully acknowledged. A first version of this paper was finished while Horst was visiting the Hausdorff Research Institute for Mathematics. Grateful acknowledgement is made for hospitality.}}

\author{Paulwin Graewe$^{1}$ \and Ulrich Horst$^{2}$ \and Jinniao Qiu$^{1}$}
\date{}

\footnotetext[1]{Department of Mathematics, Humboldt-Universit\"at zu Berlin, Unter den Linden~6, 10099~Berlin, Germany. \textit{E-mail}: \texttt{graewe@math.hu-berlin.de} (P.\ Graewe), \texttt{qiujinn@gmail.com} (J.\ Qiu).}
\footnotetext[2]{Department of Mathematics and School of Business and Economics, Humboldt-Universit\"at zu Berlin, Unter den Linden~6, 10099~Berlin, Germany.
\textit{E-mail}: \texttt{horst@math.hu-berlin.de}; corresponding author.}

\maketitle

\begin{abstract}
We establish existence, uniqueness and regularity of solution results for a class of backward stochastic partial differential equations with singular terminal condition. The equation de- scribes the value function of non-Markovian stochastic optimal control problem in which the terminal state of the controlled process is pre-specified. The analysis of such control problems is motivated by models of optimal portfolio liquidation.
\end{abstract}

%\bigskip



\section{Introduction and model formulation}

	We consider a class of non-Markov stochastic optimal control problems with a singular terminal state constraint on the controlled process. In a Markovian framework such constraints lead to nonlinear partial differential equations (PDEs) with a singularity at the terminal time. Existence and uniqueness of smooth solutions to such PDEs has recently been established in~\cite{GraeweHorstSere13}. This paper extends their results beyond the Markovian framework. We show that  the value function of the corresponding non-Markovian control problem can be characterized by a backward stochastic partial differential equation (BSPDE) with a singular terminal value. Our main contribution is to prove  existence and uniqueness of a sufficiently regular solution to this BSPDE from which one can deduce the optimal control in feedback form.

	The analysis of optimal control problems with singular state constraints on the terminal value of the controlled process is motivated by models of optimal portfolio liquidation under price-sensitive market impact. Traditional financial market models assume that price fluctuations follow some exogenous stochastic process and that all trades can be carried out at the prevailing market price. This assumption that all trades can be settled without impact on market dynamics is appropriate for small investors that trade only a negligible proportion of the average daily trading volume. It is not always appropriate for institutional investors that need to close large positions over short time periods.

	Models of optimal portfolio liquidation have received considerable attentions in the mathematical finance and stochastic control literature in recent years, see, e.g., \cite{AlmgrenChriss00,AnkirchnerKruse12,Forsyth2012,GatheralSchied11,HorstNaujokat13,Kratz13,KratzSchoeneborn13,Schied13}. The literature on optimal liquidation has so far been confined to Markovian models, where the cost functions are either deterministic or driven by stochastic factors that follow a Markovian dynamics. In real world markets, the cost of trading is often of a non-Markovian nature, though. For instance, trading costs are computed based on \textit{volume weighted average prices} (VWAP), a weighted average of past prices and volumes traded at that prices. This calls for a general mathematical framework which allows for non-Markovian factor dynamics and explicit functional dependencies of the optimal liquidation strategies on the observable factor process. 

	This paper provides such a framework.  Our primary focus is on BSPDEs with singular terminal values arising in models of optimal portfolio liquidation. Our model is flexible enough to allow for a non-Markovian factor dynamics and cost functional and for simultaneous submission of active orders (for immediate execution) to a primary market and passive block trades (for possible future execution) to a crossing network or dark pool. Dark pools are alternative trading venues that allow investors shield their orders from public view and hence to reduce market impact and trading costs. Since orders submitted to a dark pool are not openly displayed, order execution is uncertain and often modeled by a point process. To the best of our knowledge \cite{HorstNaujokat13, KratzSchoeneborn13} were the first to study portfolio liquidation problems with dark pools in continuous time.

\subsection{Model and problem formulation}

	Throughout this paper, we work on a probability space $(\Omega,\bar{\sF},\bP)$ equipped with a filtration $\{\bar{\sF}_t\}_{0 \leq t \leq T}$ that satisfies the usual conditions of completeness and right-continuity. The probability space carries two independent $m$-dimensional\footnote{The Brownian motions may well have different dimensions; this assumption is made for convenience only.} Brownian motions $W$ and $B$ as well as an independent point process $\tilde{J}$ on on a non-empty Borel set $\mathcal{Z}\subset\bR^l$ with finite characteristic measure $\mu(dz)$. We endow the set $\mathcal{Z}$ with its Borel $\sigma$-algebra $\mathscr{Z}$ and denote by $\pi(dz,dt)$ the associated Poisson random measure. The filtration generated by $W$, together with all $\bP$ null sets, is denoted by $\{\sF_t\}_{t\geq 0}$. The $\sigma$-algebra of the predictable sets on $\Omega\times[0,+\infty)$ associated with $\{\sF_t\}_{t\geq0}$ is denoted by $\sP$.

	In this work, we address the following stochastic optimal control problem with a terminal state constraint:
\begin{equation} \label{min-contrl-probm}
  \min_{\xi,\rho} E\int_0^T\left\{\eta_s(y_s)|\xi_s|^2+\lambda_s(y_s)|x_s|^2 + \int_{\mathcal {Z}} \gamma_s(y_s,z)|\rho_s(z)|^2\,\mu(dz)\right\}ds
\end{equation}
subject to
\begin{equation} \label{state-proces-contrl}
	\left\{\begin{aligned}
		&x_t=x-\int_0^t\xi_s\,ds-\int_0^t\int_{\mathcal {\mathcal {Z}}}\rho_s(z)\,\pi(dz,ds), \quad t\in[0,T];\\
		& x_T=0;\\
		& y_t=y+\int_0^tb_s(y_s, \omega)\,ds+\int_0^t\bar{\sigma}_s(y_s, \omega)\,dB_s +\int_0^t\sigma_s(y_s, \omega)\,dW_s.
	\end{aligned}\right.
\end{equation}
Here, the real-valued process $(x_t)_{t\in[0,T]}$ is the \textit{state process}; in a portfolio liquidation framework $x_t$ describes the number of shares held at time $t \in [0,T]$. The state process is governed by a pair of \textit{controls} $(\xi,\rho)$ describing, for instance, the rates at which the portfolio is liquidated in the primary market and the block trades placed in the dark pool, respectively, with the Poisson random measure $\pi$ governing dark pool executions. 

The $d$-dimensional process $(y_t)_{t\in[0,T]}$ is an uncontrolled \textit{factor process}. The factor process is driven by the Wiener processes $W$ and $B$; the coefficients $b_t(y;\omega), \bar \sigma_t(y;\omega)$ and $\sigma_t(y;\omega)$ are $\sF$-adapted. We sometimes write $x^{s,x,\xi,\rho}_t$ for $0\leq s\leq t\leq T$ to indicate the dependence of the state process on the control $(\xi,\rho)$, the initial time $s \in [0,T]$ and initial state $x\in \mathbb{R}$. Likewise, we sometimes write $y^{s,y}_t$. The set of \textit{admissible controls} consists of all pairs $(\xi,\rho)\in\cL^2_{\bar{\sF}}(0,T; \bR)\times \cL^{2}_{\bar{\sF}}(0,T;L^2(\mathcal{Z}))$ that satisfy almost surely the \textit{terminal state constraint}
\begin{equation} \label{terminal-condition}
	x_{T}=0.
\end{equation}

	We assume that the cost associated with an admissible control $(\xi,\rho)$ at time~$t \in[0,T)$ and state $(x,y)\in\mathbb R\times \mathbb R^d$ is given by
\begin{equation*} %\label{cost0}
	J_t(x,y;\xi,\rho):=E^{\bar \sF_t} \int_t^T\left\{\eta_s(y_s^{t,y})|\xi_s|^2+\lambda_s(y_s^{t,y})|x_s^{t,x;\xi,\rho}|^2  +\int_{\mathcal {Z}}\gamma_s(y_s^{t,y},z)|\rho_s(z)|^2\,\mu(dz)\right\}ds
\end{equation*}
for $\sF$-adapted coefficients $\eta_t(y;\omega), \lambda_t(y;\omega)$ and $\gamma_t(y;\omega)$. The \textit{value function} is denoted by
\begin{equation} \label{value-func}
  V_t(x,y):=\essinf_{(\xi,\rho) \text{ admissible}} J_t(x,y;\xi,\rho)
\end{equation}
In a portfolio liquidation framework the coefficients $\eta_t(y;\omega)$ and $\lambda_t(y;\omega)$ measure the market impact costs and the investor's desire for early liquidation (``risk aversion''), respectively. The term $\gamma_t(y;\omega)$ measures the so-called \textit{slippage} or \textit{adverse selection costs} associated with the execution of dark pool orders.\footnote{The notion of ``slippage costs'' refers to the costs associated with an adversely executed order, e.g., a buy order execution in a dark pool immediately before a price decrease in the primary market.} $V_t(x,y)$ is the cost of liquidating the portfolio comprising $x$ shares during the time interval $[t,T]$, given the current value $y$ of the factor process and (\ref{terminal-condition}) reflects the fact that full liquidation is required by the terminal time.

\subsection{The BSPDE for the value function}

	The special case where $\eta, \lambda$ and~$\gamma$ are independent of $y$ has recently been analyzed by Ankirchner~et~al.~\cite{AnkirchnerJeanblancKruse13}. In this case, the value function can be described by a backward stochastic differential equation~(BSDE) with singular terminal value. To the best of our knowledge, such equations were first analyzed by Popier~\cite{Popier06}. A class of stochastic optimal control problems with the terminal states being constrained to a convex set were studied by Ji and Zhou~\cite{JiZhou-2008} using forward-backward stochastic differential systems. They assumed a strict invertibility of the diffusion term with respect to the control and applied a maximum principle of Pontryagin type. We solve the control problem by solving the corresponding stochastic Hamilton-Jacobi-Bellman (HJB) equation introduced by Peng~\cite{Peng_92} for non-Markovian control problems.

	In view of the linear-quadratic structure of the cost functional a standard arguments suggest a multiplicative decomposition of the value function of the form
\begin{equation} \label{ansatz}
  V_t(x,y)=u_t(y)x^2 \quad \textrm{and}\quad \Psi_t(x,y)=\psi_t(y)x^2
\end{equation}
for a pair of adapted processes $(u,\psi)$ that satisfies the BSPDE (in a suitable class of stochastic processes)
\begin{equation} \label{BSPDE-singlr}
	\left\{\begin{aligned}
  	-du_t(y) & =\left\{\mathcal L u_t(y)+\mathcal M\psi_t(y)+F(s,y,u_t(y)) \right\} dt-\psi_t(y)\, dW_{t}, \quad (t,y)\in [0,T]\times \bR^d;\\
    u_T(y) & = +\infty,  \quad y\in\bR^d,
	\end{aligned}\right.
\end{equation}
where, for $a:=\frac{1}{2}(\sigma\sigma^{\mathcal {T}}+\bar{\sigma}\bar{\sigma}^{\mathcal{T}})$, the operators $\cL$ and $\cM$ act on twice, respectively once continuously differentiable functions according to
\begin{equation*}
	\cL u_t(y) = \tr\left( (a_t(y) D^2 u_t(y) \right) + b^{\mathcal{T}}_t(y)D u_t(y) \quad \text{and} \quad
	\cM \psi_t(y) = \tr\left(  D \psi_t(y)\sigma^{\mathcal{T}}_t(y) \right)
\end{equation*}
with $D$ and $D^2$ being the gradient operator and Hessian matrix respectively throughout this work, and the non-linearity $F: \bR_+\times \bR^d\times L^{0}(\bR^d) \to \bR$ is given by
\begin{equation*}
  F(t,y,\phi(y)) := \lambda_t(y)-\int_{\mathcal{Z}}\frac{|\phi(y)|^2}{\gamma_t(y,z)+\phi(y)}\,\mu(dz) -\frac{|\phi(y)|^2}{\eta_t(y)}.
\end{equation*}

	The preceding BSPDE depends quadratically on $u_t(y)$. Although BSPDEs have been extensively studied in the applied probability and financial mathematics literature, see, e.g., \cite{Bensoussan-fbsd-1983,ChangPangYong-2009,Du_Zhang_DegSemilin2012,EnglezosKaratzas09,Mania-Tevzaze-2003,TangZhang-2009}, no general theory for BSPDEs which are of quadratic growth in $u$ is yet available, not even for finite terminal values. 

	Using recent existence of solutions results for nonlinear BSPDEs \cite{Qiu-2012,QiuTangMPBSPDE11,QiuWei-RBSPDE-2013,Tang-Yang-2011} and the \textrm{I}t$\hat{\textrm{o}}$-\textrm{W}entzell formula for distribution-valued processes \cite{Krylov_09,Tang-Yang-2011} we first prove that the BSPDE resulting from a corresponding control problem with finite terminal condition has a sufficiently smooth solution. Subsequently, we establish a comparison principle from which we deduce that the solution to the BSPDE with infinite terminal value can be obtained as the limit of an increasing sequence of solutions to BSPDEs with finite terminal conditions. We also obtain an explicit asymptotic property of the solution $u$ near the terminal time.

	When all the coefficients are deterministic functions of the state and control variables, then we are in the Markovian setting. In this case our BSPDE simplifies to a parabolic PDE (to be understood in the distributional sense). As a byproduct of our general existence and uniqueness result, corresponding results are obtained under weak assumptions on the model parameters in the Markovian framework.

	The remainder of this paper is organized as follows. Our main assumptions and results are summarized in Section~2. Section~3 is devoted to the proof of the verification theorem while Section~4 establishes the existence of the solution for our singular BSPDE that satisfies the assumptions of the verification theorem. In Section~5 we prove that the BSPDE~\eqref{BSPDE-singlr} actually has a unique non-negative solution in a larger class of stochastic processes that automatically satisfies the asymptotic behavior around the terminal time that is needed for the proof of the verification theorem. The appendix recalls three results on BSPDEs which are used throughout this work.

\section{The main results}

	In order to state our main result we need to introduce some function spaces. For a Banach space $V$ we denote by $\cS^p_{\sF} ([0,T];V)$, $p\in[1,\infty)$, the set of all the $V$-valued and $\sP$-measurable c\`adl\`ag processes $(X_{t})_{t\in [0,T]}$ such that
\[
	\|X\|_{\cS _{\sF}^p([0,T];V)}^p= E \sup_{t\in [0,T]} \|X_t\|_V^p< \infty.
\]
By $\cL^p_{\sF}(0,T;V)$ we denote the class of $V$-valued $\sP$-measurable processes $(u_t)_{t\in[0,T]}$ such that
\begin{align*}
	\|u\|^p_{\cL^p_{\sF}(0,T;V)} &=E \int_0^T\|u_t\|^p_{V}\,dt<\infty, \quad p\in [1,\infty);\\
 \|u\|_{\cL^{\infty}_{\sF}(0,T;V)}&= \esssup_{(\omega,t)\in\Omega\times[0,T]}\|u_t\|_{V}<\infty, \quad p=\infty.
\end{align*} 
In a similarly way we define $\cS _{\bar{\sF}}^p([0,T];V)$ and $\cL^p_{\bar{\sF}}(0,T;V)$. For $u\in\cL^p_{\sF}(0,T;L^p(\mathbb R^d))$, $p\in[1,\infty)$, we write $u\in\cL^{p,\infty}_{\sF}(0,T)$ if
\begin{enumerate}[(i)]
	\item $u$ is continuous on $[0,T]$, $\mathbb{P}\otimes dx$-a.e.;
	\item $\displaystyle\|u\|_{\cL^{p,\infty}_{\sF}(0,T)}^p= E\int_{\bR^{d}}\sup_{t\in[0,T]}|u(t,x)|^{p}\,dx<\infty.$
\end{enumerate}
As usual, the Sobolev space of all functions whose first $k$ derivatives belong to $L^p(\Pi)$ for some domain $\Pi\subset\bR^d$  is denoted by $H^{k,p}(\Pi)$. For simplicity, by saying a finite dimensional space-valued function $u=(u_1,\ldots,u_l)\in H^{k,p}(\Pi)$, $l\in\mathbb{N}$, we mean $u_1,\ldots,u_l\in H^{k,p}(\Pi)$ and $\|u\|_{H^{k,p}(\Pi)}^p:=\sum_{j=1}^l\|u_j\|_{H^{k,p}(\Pi)}^p$.

	Throughout this work, we use $\langle\cdot,\,\cdot\rangle$ to denote the inner product in the usual Hilbert space $L^2(\bR^d)=H^{0,2}(\bR^d)$. For $k\in\bN_0$, we set 
\[
	\mathcal{H}^k=\cS^2_{\sF}([0,T];H^{k,2}(\bR^d))\cap \cL^2_{\sF}(0,T;H^{k+1,2}(\bR^d)) 
\]	
equipped with the norm
\begin{equation*}
	\|u\|_{\mathcal{H}^k}^2= \|u\|^2_{\cS^2_{\sF}([0,T];H^{k,2}(\bR^d))}+\|u\|^2_{\cL^2_{\sF}(0,T;H^{k+1,2}(\bR^d))}.
\end{equation*}

	Our goal is to prove existence of a sufficiently regular solution to the BSPDE (\ref{BSPDE-singlr}) and to characterize the value function of our control problem in terms of that solution. To this end, we first define what we mean by a solution to (\ref{BSPDE-singlr}). 

\begin{definition} \label{defn-solution}
	A pair of processes $(u,\psi) $ is a solution to the BSPDE (\ref{BSPDE-singlr}) if for all $0 \leq t < \tau <T$ it holds $(u,\psi)1_{[0,\tau]\times\mathcal O}\in\cL^2_{\sF}(0,\tau;H^{2,2}(\cO))\times\cL^2_{\sF}(0,\tau;H^{1,2}(\cO))$ for all bounded balls $\cO\subset\bR^d$,
\begin{equation*}
	u_t(y) = u_\tau(y) + \int_t^\tau \left\{ \mathcal  L u_s(y)+\mathcal M\psi_s(y)+F(s,y,u_s(y)) \right\} ds -\int_t^\tau\psi_s(y)\, dW_{t}, \quad dy\textrm{-a.e.},
\end{equation*}
and
\begin{equation*}
	\lim_{\tau\uparrow T} u_{\tau}(y)=+\infty, \quad \mathbb{P}\otimes dy\text{-a.e.}
\end{equation*}
%\begin{enumerate}[(i)]
%	\item $(u,\psi)1_{[0,\tau]\times\mathcal O}\in\cL^2_{\sF}(0,\tau;H^{2,2}(\cO))\times\cL^2_{\sF}(0,\tau;H^{1,2}(\cO))$ for all bounded open balls $\cO\subset\bR^d$,
%	\item $\displaystyle u_t(y) = u_\tau(y) + \int_t^\tau \left\{ \mathcal  L u_s(y)+\mathcal M\psi_s(y)+F(s,y,u_s(y)) \right\} ds -\int_t^\tau\psi_s(y)\, dW_{t}, \quad dy\textrm{-a.e.},$
%	\item $\lim_{\tau\uparrow T} u_{\tau}(y)=+\infty, \quad \mathbb{P}\otimes dy\text{-a.s.}$
%\end{enumerate}
\end{definition}

Our results are established under the following standard measurability and regularity conditions on the model parameters:
\begin{enumerate}[$({\mathcal A} 1)$]
	\item  The function
\begin{equation*}
	(b,\sigma,\bar{\sigma},\eta,\lambda): \Omega\times[0,T]\times\bR^d\longrightarrow \bR^d\times \bR^{d\times m}\times \bR^{d\times m}\times \bR_+\times \bR_+
\end{equation*}
is $\sP\times \mathscr{B}(\bR^d)$-measurable and essentially bounded by $\Lambda>0$. Moreover,
$$
	\gamma: \Omega\times[0,T]\times\bR^d\times\mathcal{Z}\longrightarrow [0,+\infty],
$$
is $\sP\times \mathscr{B}(\bR^d)\times \mathscr{Z}$-measurable.
	\item There exists a constant $L$ such that for all $y_1,y_2\in \bR^d$ and $(\omega,t)\in \Omega\times [0,T]$,
\begin{equation*}
	|b_t(y_1)-b_t(y_2)|+|\sigma_t(y_1)-\sigma_t(y_2)|+|\bar{\sigma}_t(y_1)-\bar{\sigma}_t(y_2)|\leq L|y_1-y_2|.
\end{equation*}
	\item There exist positive constants $\kappa$ and $\kappa_0$ such that for all $(y,\xi,t)\in\bR^d\times\bR^d\times[0,T]$,
$$
\sum_{i,j=1}^d\sum_{r=1}^m\bar{\sigma}^{ir}_t(y)\bar{\sigma}^{jr}_t(y)\xi^i\xi^j \geq \kappa |\xi|^2\quad
\textrm{and}\quad
\eta_t(y)\geq \kappa_0,\quad \mathbb P\text{-a.e.}
$$
\end{enumerate}

%%%%%%%%%%%%%%%%%%%%%%%%%%%%%%%%%%
%%%%%%%%%%%%%%%%%%%%%%%%%%%%%%%%%%
%%%%%%%%%%%%%%%%%%%%%%%%%%%%%%%%%%

The verification theorem requires an integral representation of the process
\begin{equation} \label{def-Phi}
	\left\{ u_t(y_t^{0,y}) | x_t^{0,x,\xi,\rho} |^2 \right\}_{0 \leq t \leq T}.
\end{equation}
We are unaware of a general $L^{\infty}$-theory for BSPDEs; at the same time,  under assumptions $(\mathcal{A}1)-(\mathcal{A}3)$, we can not apply the existing $L^p$-theory $(p\in (1,\infty))$ in our framework directly; see \cite{DuQiuTang10} and references therein. Moreover, as it will turn out, the solution $u$ to (\ref{BSPDE-singlr}) has to be regular enough to allow for an application of the generalized It\^o-Kunita-Wentzell formula of Tang and Yang \cite{Tang-Yang-2011} to the composition~$u_t(y_t)$. To guarantee  regularity and apply the existing $L^p$-theory on BSPDEs, we work with a weighted solution. More precisely, we define, for any integer $q>d$, the function
\[
	\theta: \mathbb{R}^d \to \mathbb{R}, \quad y \mapsto (1+|y|^2)^{-q},
\]
and analyze ${\theta}u$ instead of $u$. A direct computation verifies that $(u,\psi)$ is a solution to \eqref{BSPDE-singlr} if and only if $({\theta}u,{\theta}\psi)$ solves
\begin{equation}\label{BSPDE-singlr-1}
  \left\{\begin{aligned}
  	-dv_t(y) & =\{\tilde{\mathcal L} v_t(y)+\tilde{\mathcal M}\zeta_t(y)+{\theta}F(t,y,(\theta^{-1} v_t)(y))\}\,dt -\zeta_t(y)\, dW_{t}, \quad (t,y)\in [0,T)\times \bR^d;\\
    v_T(y)  &= +\infty,  \quad y\in\bR^d,
	\end{aligned}\right.
\end{equation}
where
\begin{equation*}
	\tilde{ \mathcal L} v_t(y) := \tr(a_t(y)D^2 v_t(y))+ \tilde{b}^{\mathcal{T}}_t(y)D v_t(y)+c_t(y)v_t(y)
\end{equation*}
and
\begin{equation*}
	\tilde{\mathcal M}\zeta_t(y) := \tr(D \zeta_t(y)\sigma^{\mathcal{T}}_t(y))+\beta_t^{\mathcal{T}}(y)\zeta_t(y)
\end{equation*}
and the functions $\tilde{b}_t=(\tilde{b}^i_t)_{i=1}^d$, $\beta_t = (\beta^r_t)_{r=1}^m$ and $c_t$ are given by
\begin{align*}
	\tilde{b}_t^i(y) :&=b^i_t(y) +\frac{4q}{1+|y|^2}\sum_{j=1}^d a^{ij}_t(y)y^j,\\
  \beta^r_t(y) &:= \frac{2q}{1+|y|^2}\sum_{j=1}^d\sigma_t^{jr}(y)y^j, \\
	c_t(y) &:=\frac{2q}{1+|y|^2}\bigg(\tr(a_t(y)) +\sum_{i=1}^dy^ib^i_t(y) +\frac{2(q-1)}{1+|y|^2} \sum_{i,j=1}^da_t^{ij}(y)y^iy^j\bigg).
\end{align*}

For each $\delta\in (0,1)$, let $C^{\delta}(\bR^d)$ be the usual H\"{o}lder space on $\bR^d$. We are now ready to summarize the main results of this paper. 

\begin{theorem} \label{thm-main}
Under assumptions $(\mathcal{A}1)-(\mathcal{A}3)$ the following holds:
\begin{enumerate}[(i)]
	\item The  BSPDE \eqref{BSPDE-singlr} admits a solution $(u,\psi)$ which satisfies  
\begin{equation} \label{solution-1}
	({\theta}u,{\theta}\psi)1_{[0,\tau]}\in \mathcal{H}^1\times  \cL^2_{\sF}(0,T;H^{1,2}(\bR^d)), \quad \tau\in[0,T), 
\end{equation}
and
\begin{align*}
    \frac{c_0}{T-t}\leq u_t(y) \leq \frac{c_1}{T-t},\quad \mathbb{P}\otimes dt \otimes dy\mbox{-a.e.}, %\label{eq-thm-2-u}
\end{align*}
  with $c_0$ and $c_1$ being two  positive constants. The function 
\begin{align} \label{VF}
    V(t,y,x):=u_t(y)x^2,\quad (t,x,y)\in[0,T]\times \bR\times \bR^d,
\end{align}
coincides with the value function for almost every $y\in\bR^d$ and the optimal (feedback) control is given by
\begin{equation*}
    \left(\xi^{*}_t,\ \rho^*_t(z)\right)=
    \left(\frac{u_t(y_t)x_t}{\eta_t(y_t)},\  \frac{u_t(y_t)x_{t-}}{\gamma_t(z,y_t)+u_t(y_t)}   \right). %\label{eq-thm-control}
\end{equation*}
 \item The solution $(u,\psi)$ is the unique non-negative solution to \eqref{BSPDE-singlr} in that sense that if $(\bar{u},\bar{\psi})$ is another solution satisfying (\ref{solution-1}) and $\bar{u}\geq 0$, $\mathbb{P}\otimes dt \otimes dx$-a.e., then
  $$
  \bar{u}_t(y)= u_t(y),\quad \mathbb{P}\otimes dt \otimes dy\mbox{-a.e.}
  $$
%\item[(iii)] The solution $(u, \psi)$ is the unique solution that satisfies the following additional conditions: 
%	\begin{itemize}
%		\item[$(\mathcal{C}1)$] $({\theta} u,{\theta}\psi)1_{[0,t]}\in  \mathcal H^1\times \mathcal L^2_{\sF}(0,T;H^{1,2}),\quad \forall\,\,t\in[0,T)$
%		\item[$(\mathcal{C} 2)$] $u\in\cL_{\sF}^\infty(0,t;L^\infty(\mathbb R^d)) ,\quad \forall\,\,t\in[0,T)$
%		\item[$(\mathcal{C} 3)$] $u\geq0, \quad\mathbb P\otimes dt \otimes dy$-a.e..
%		%:=\cap_{s\in(0,T)}\cL_{\sF}^\infty(0,s;L^\infty(\mathbb R^d))$
%		\end{itemize}
\item Under the additional assumption that $\sigma$ is spatially invariant, i.e., does not depend on $y$
one has furthermore for any $p\in(2,+\infty)$,
\[
	\theta(\cdot) u_{\cdot}\Big(\cdot+\int_0^{\cdot}\sigma_s\,dW_s\Big) 
	\in\bigcap\limits_{\tau\in(0,T)}\bigcap\limits_{\delta\in(0,1)}\cL^{2,\infty}_{\sF}(0,\tau)\cap   \cS^p_{\sF}([0,\tau];C^{\delta}(\bR^d))
\]
and the function $V(t,y,x)$ in \eqref{VF} coincides with the value function for \textit{every} $y\in\bR^d$. 
\end{enumerate}
\end{theorem}
\vspace{.7em}
\begin{remark}
When all the coefficients $b,\sigma,\bar{\sigma},\lambda,\eta,\gamma$ are deterministic functions, then the optimal control problem is Markovian and the corresponding BSPDE \eqref{BSPDE-singlr} reduces to a deterministic parabolic partial differential equation
\begin{equation}\label{BSPDE-singlr-sec6}
  \left\{\begin{aligned}
  		-\partial_t u &=\mathcal L u+F(t,y,u), \quad (t,y)\in [0,T]\times \bR^d;\\
    u_T(y) &= +\infty,  \quad y\in\bR^d.
  \end{aligned}\right.
\end{equation}
%Without any loss of generality, we assume $\sigma\equiv 0$. 
In this case, we may with no loss of generality assume that $\sigma \equiv 0$ so Theorem~\ref{thm-main}~(iii) indicates that~\eqref{BSPDE-singlr-sec6} admits a unique non-negative solution $u$ in the distributional sense that satisfies 
\[
	{\theta}u\in\bigcap\limits_{\tau\in(0,T)}\bigcap\limits_{\delta\in(0,1)}  C([0,\tau];C^{\delta}(\mathbb{R}^d)),
\] and $V(t,y,x)=u_t(y)x^2$ coincides with the continuous value function for every $y\in\mathbb{R}^d$. 
\end{remark}
%%%%%%%%%%%%%%%%%%%%%%%%%%%%%%%%%%%%%%%%
%%%%%%%%%%%%%%%%%%%%%%%%%%%%%%%%%%%%%%%%
%%%%%%%%%%%%%%%%%%%%%%%%%%%%%%%%%%%%%%%%

\section{The verification theorem}

We are now ready to state the verification theorem. Its proof requires some preparation and is carried out below. 

\begin{theorem}\label{thm-verification}
  Let assumptions $(\mathcal{A}1)-(\mathcal{A}3)$ be satisfied and suppose that $(u,\psi)$ is a solution to \eqref{BSPDE-singlr} that satisfies 
\begin{equation} \label{ABC}
  	({\theta}u,{\theta}\psi)1_{[0,t]}\in\mathcal{H}^1 \times \cL^2_{\sF}(0,T;H^{1,2}(\bR^d)),\quad t\in[0,T),
\end{equation}	
and
\begin{equation} \label{eq-thm-u}
    \frac{c_0}{T-t}\leq u_t(y) \leq \frac{c_1}{T-t},\quad \mathbb{P}\otimes dt \otimes dy\text{-a.e.}, 
\end{equation}
with $c_0$ and $c_1$ being two  positive constants. Then, ${\theta}u\in\cap_{\tau\in(0,T)}\cL^{2,\infty}_{\sF}(0,\tau)$ and
\begin{equation*}
    V(t,y,x):=u_t(y)x^2,\quad (t,x,y)\in[0,T]\times \bR\times \bR^d,
\end{equation*}
coincides with the value function of \eqref{value-func} for almost every $y\in\bR^d$. Moreover, the optimal (feedback) control is given by
\begin{equation} \label{eq-thm-control}
  \left(\xi^{*}_t,\, \rho^*_t(z)\right)= \left(\frac{u_t(y_t)x_t}{\eta_t(y_t)},\, \frac{u_t(y_t)x_{t-}}{\gamma_t(z,y_t)+u_t(y_t)}   \right).
\end{equation}
\end{theorem}
	
	We first recall the following generalized It\^o-Kunita-Wentzell formula from which we later derive an integral representation for $(\ref{def-Phi})$.

\begin{lemma}[{\cite[Theorem 3.1]{Tang-Yang-2011}}] \label{lem-Yang-Tang-13}
  Let the coefficients $b$, $\sigma$ and $\bar{\sigma}$ satisfy the assumptions $(\mathcal{A}1)-(\mathcal{A}3)$ and let $G\in L^2(\Omega,\sF_T;H^{1,2}(\bR^d))$, $\Phi\in\cL^2_{\sF}(0,T;H^{2,2}(\bR^d))$, $\Upsilon\in\cL^2_{\sF}(0,T;H^{1,2}(\bR^d))$ and $F\in\cL^2_{\sF}(0,T;L^2(\bR^d))$ such that
\[
	\Phi_t(y)=G(y)+\int_t^TF_s(y)\,ds-\int_t^T\Upsilon_s(y)\,dW_s, \quad dy\text{-a.e.},\quad \text{for all } t\in[0,T].
\]
Then, the compositions $\Phi_\cdot(y_{\cdot}^{s,\cdot})$, $G(y_T^{s,\cdot})$, $F_\cdot(y_{\cdot}^{s,\cdot})$ and $\Upsilon_\cdot(y_{\cdot}^{s,\cdot})$ are well-defined under the measure $\mathbb{P}\otimes dt \otimes dy$, and for almost every $y\in\bR^d$ it holds almost surely for all $t\in[s,T]$
\begin{multline*}
	\Phi_t(y_t^{s,y}) = G(y_T^{s,y}) -\int_t^T \Big\{\tr\left( a_r(y_r^{s,y})D^2 \Phi_r(y_r^{s,y}) + D\Upsilon_r(y_r^{s,y})\sigma^{\mathcal{T}}_r(y_r^{s,y})\right)+b^{\mathcal{T}}_r(y_r^{s,y})D \Phi_r(y_r^{s,y})\\-F_r(y_r^{s,y})\Big\}\,dr-\int_t^T\left\{\sigma_r^{\mathcal{T}}(y_r^{s,y})D \Phi_r(y_r^{s,y})+\Upsilon_r(y_r^{s,y})\right\}dW_r-\int_t^T\bar{\sigma}_r^{\mathcal{T}}(y_r^{s,y})D \Phi_r(y_r^{s,y})\,dB_r.
\end{multline*}
\end{lemma}

	Using local estimates for the weak solutions to BSPDEs from \cite{QiuTangMPBSPDE11}, Yang and Tang \cite{Tang-Yang-2011} proved that the above compositions 
%$\Phi(\cdot,y_{\cdot}^{s,\cdot})$, $G(y_T^{s,\cdot})$ and $F(\cdot,y_{\cdot}^{s,\cdot})$ 
are well defined. But they did not establish the integrability properties needed for our proof of the verification theorem. The following corollary establishes such properties. The proof is purely technical and postponed to the appendix.

\begin{corollary}	\label{cor-ito-wentzell}
	Under the hypothesis of Lemma~\ref{lem-Yang-Tang-13}, $\Phi_{\cdot}(y_{\cdot}^{s,y})$ is a continuous and uniformly integrable semi-martingale for almost every $y\in\bR^d$ and $\Phi\in\cL^{2,\infty}_{\sF}(0,T)$. Furthermore, there exists a constant $C$ that depends only on $\kappa$, $L$, $\Lambda$ and $T$ such that
\begin{enumerate}[(i)]
	\item $\displaystyle\int_{\bR^d} \left|G(y_T^{s,y})\right|^2dy  \leq C \|G\|^2_{L^2(\Omega,\sF_T;H^{1,2}(\bR^d))};$
	\item $\displaystyle\int_{\bR^d}\bigg(\int_s^T E \left|F_r(y_r^{s,y})\right|dr\bigg)^2dy \leq C \|F\|^2_{\cL^2_{\sF}(s,T;L^2(\bR^d))};$
	\item $\displaystyle\int_{\bR^d}\sup_{r\in[s,T]}E \left|\Phi_r(y_r^{s,y})\right|^2dy\leq C\left(\|G\|^2_{L^2(\Omega,\sF_T;H^{1,2}(\bR^d))}  +\|F\|^2_{\cL^2_{\sF}(s,T;L^2(\bR^d))}\right).$
\end{enumerate}
\end{corollary}

	Our second auxiliary result is the following lemma on the set of admissible controls. It states that we may with no loss of generality assume that the state process associated with an admissible control is monotone. A similar result has been established in \cite{GraeweHorstSere13} for the Markovian case.

\begin{lemma}\label{lem-admissible-control}
  For each admissible control $(\xi,\rho)$ there exists a corresponding admissible control $(\hat{\xi},\hat{\rho})$ with lesser or equal cost such that the process 
  $x^{0,x;\hat{\xi},\hat{\rho}}$ is almost surely monotone. Furthermore, there exists a constant $C < \infty$ which is independent of $t,x,\hat{\rho},\hat{\xi}$ such that
\begin{equation}  \label{est-lem-adm-control}
|x_t^{0,x;\hat{\xi},\hat{\rho}}|^2 \leq C (T-t) E^{\bar\sF_t}\int_t^T|\hat{\xi}_s|^2\,ds, \quad  t\in[0,T].
\end{equation}
%with the positive constant $C$ being independent of $t,x,\hat{\rho}$ and $\hat{\xi}$.
\end{lemma}

\begin{proof}
  Assume that $x\geq 0$ (the case for $x\leq 0$ follows in a similar way). For the admissible control $(\xi,\rho)$, let $(\tilde{x}_t)\in \cS _{\bar{\sF}}^2([0,T])$ be the unique solution of the following stochastic differential equation
\[
	\tilde{x}_t=x-\int_0^t\xi^+_s\,ds-\int_0^t\int_{\mathcal {\mathcal {Z}}}\rho^+_s(z)\wedge \tilde{x}_s^+\,\pi(dz,ds),
\]
where $f^+:=\max\{f,0\}$ for $f=\tilde{x}_s,\xi_s$ or $\rho_s$. Set
\[
	\hat{\xi}_t :=\xi_t^+ 1_{\tilde{x}_t>0} \quad \textrm{and} \quad\hat{\rho}_t(z) :=\rho^+_t(z)\wedge \tilde{x}_s^+.
\]
It is easy to check that $(\hat{\xi},\hat{\rho})\in \cL^2_{\bar{\sF}}(0,T)\times \cL^{2}_{\bar{\sF}}(0,T;L^2(\mathcal{Z}))$ is an admissible control pair with lesser or equal cost and that $x^{0,x;\hat{\xi},\hat{\rho}}$ is decreasing almost surely. Since $x^{0,x;\hat{\xi},\hat{\rho}}$ is non-negative and decreasing,
\[
	0\leq\hat{\rho}_t\leq x_t^{0,x;\hat{\xi},\hat{\rho}},\quad \mathbb P\otimes dt\otimes \mu(dz)\text{-a.e.} \]
Thus,
\begin{align*}
	|x_t^{0,x;\hat{\xi},\hat{\rho}}|^2 %\leq 2 E^{\bar\sF_t}\left\{\left|\int_t^T\hat{\xi}_s\,ds\right|^2+\left| \int_{[t,T]\times\mathcal{Z}}\hat{\rho}_{s-}(z)\,\pi(dz,ds)\right|^2\right\}\\
  %&\quad = 2 E^{\bar\sF_t}\left\{ \left|\int_t^T\hat{\xi}_s\,ds\right|^2+\left| \int_{[t,T]\times\mathcal{Z}}\hat{\rho}_{s-}(z)\,\left(\tilde{\pi}(dz,ds)-\mu(dz)ds\right)\right|^2\right\}\\
  &\leq C E^{\bar\sF_t}\bigg\{ \bigg|\int_t^T\hat{\xi}_s\,ds\bigg|^2+\bigg| \int_{[t,T]\times\mathcal{Z}}\hat{\rho}_{s-}(z)\,\tilde{\pi}(dz,ds)\bigg|^2+\bigg|\int_{[t,T]\times\mathcal{Z}}\hat{\rho}_{s-}(z)\,\mu(dz)ds\bigg|^2\bigg\}\\
  &\leq C(T-t)E^{\bar\sF_t}\bigg\{\int_t^T|\hat{\xi}_s|^2\,ds+ \int_t^T|x_s^{0,x;\hat{\xi},\hat{\rho}}|^2\,ds\bigg\},
\end{align*}
which by Gronwall's inequality implies %(see \cite[Corollary B1, Appendix B]{DuffieEpsteinSDU92}),
\begin{equation*}
	|x_t^{0,x;\hat{\xi},\hat{\rho}}|^2  \leq C (T-t)E^{\bar\sF_t}\int_t^T|\hat{\xi}_s|^2\,ds. \qedhere
\end{equation*}
\end{proof}

We are now ready to give the proof of the verification theorem. 

{\em Proof of Theorem \ref{thm-verification}.}
	By assumption ${\theta}u1_{[0,t]}\in \mathcal{H}^1$ for any $t\in(0,T)$, an application of Proposition \ref{prop-bspde-duqiutang} with $G={\theta}u_\tau$ for any $\tau<T$ yields ${\theta}u\in\cap_{\tau\in(0,T)}\cL^{2,\infty}_{\sF}(0,\tau)$. 
	
The stochastic HJB equation associated with our optimization problem is given by the following BSPDE: 
\begin{equation*}\label{BSPDE-value-func}
	\left\{\begin{aligned}
		-&dV_t(x,y)  = \bigg[\mathcal L V_t(x,y)+\mathcal M \Psi_t(x,y) +\essinf_{\xi,\rho}\bigg\{-\xi D_xV_t(x,y)+\eta_t(y)|\xi|^2 +  \lambda_t(y)|x|^2\\ & + \int_{\mathcal{Z}} \big\{ V_t(x-\rho,y)- V_t(x,y) +\gamma_t(y,z)|\rho|^2\big\}\,\mu(dz) \bigg\}\bigg]\, dt   -\Psi_t(x,y)\, dW_{t}, \quad (t,x,y)\in [0,T)\times\bR\times \bR^d; \\
    &V_T(x,y)= +\infty\cdot1_{x\neq 0},  \quad (x,y)\in\bR\times\bR^d.
	\end{aligned}\right.
\end{equation*}
It is easy to show that the pair $V_t(x,y) := u_t(x)|x|^2$ and $\Psi_t(x,y) := \psi_t(y)|x|^2$ solves the above equation if and only if $(u,\psi)$ solves  (\ref{BSPDE-singlr}). This shows that $(\xi^*,\rho^*)$ is the candidate optimal strategy. It therefore remains to show that $(\xi^*,\rho^*)$ is admissible and attains the minimal cost. 

In order to show admissibility, we plug the explicit expression for $(\xi^*,\rho^*)$ into the state process and get
\begin{equation*} 
    x_{t}^*:=x\prod_{0<s\leq t}\left\{1-\int_{Z}\frac{u_s(y_s^{0,y})}{\gamma_s(y_s^{0,y},z)+u_s(y_s^{0,y})}\,\pi(dz,\{s\})  \right\}
    \exp\left(-\int_0^t\frac{u_s(y_s^{0,y})}{\eta_s(y_s^{0,y})}\,ds\right)
\end{equation*}
for $t\in[0,t)$. Hence,
\begin{align*}
  |x_{t}^*|\leq |x| \exp\left(-\int_0^t\frac{u_s(y_s^{0,y})}{\eta_s(y_s^{0,y})}\,ds\right)\leq|x| \exp\left(-\int_0^t \frac{c_0}{\Lambda(T-s)}\,ds \right)=|x|\left(\frac{T-t}{T}\right)^{c_0/\Lambda}\stackrel{t \uparrow T}{\longrightarrow} 0.
\end{align*}
From the definition of $(\xi^*,\rho^*)$, we immediately infer that $\rho^*\in\cL^{2}_{\bar{\sF}}(0,T;L^2(\mathcal{Z}))$ and $\xi^*\in \cL^2_{\bar{\sF}}(0,t;\bR)$ for any $t\in(0,T)$. 
%Hence $(\xi^*,\rho^*)$ is admissible. 
Moreover, the associated state sequence $x^*$ is monotone.    
	
In order to show that $(\xi^*,\rho^*)$ is admissible and that the cost functional attains its minimum at $(\xi^*,\rho^*)$, we notice that the process ${\theta}(y_t^{0,y}) u_t(y_t^{0,y})$ satisfies the assumptions of Lemma \ref{lem-Yang-Tang-13} so we can apply the generalized It\^o-Kunita-Wentzell formula. A subsequent application of the standard It\^o formula to the product of $\theta^{-1}$ and $\theta u$ yields the stochastic differential equation for $u_t(y_t^{0,y})$. 

Applying the standard It\^o formula again, this time to  $u_t(y_t^{0,y}) |x_t^{0,x;\xi,\rho}|^2$, we finally obtain the SDE for the candidate value function. A tedious but straightforward computation shows that for all admissible strategies $(\xi,\rho)$ it holds for almost every $y\in\bR^d$ that
\begin{multline}\label{xyzz}
	u_t(y)\big|x|^2- E^{\bar{\sF}_t}\left\{ u_\tau(y_\tau^{t,y})\big|x_\tau^{t,x;\xi,\rho}\big|^2\right\} \\
		\begin{aligned}
  			&
			\begin{split}
				=E^{\bar\sF_t} \int_t^\tau \bigg\{ 2u_s(y_s^{t,y})x_s^{t,x;\xi,\rho}\xi_s +u_s(y_s^{t,y})\int_{\mathcal{Z}}\big\{ 2\rho_s(z)x_s^{t,x;\xi,\rho}-|\rho_s(z)|^2\big\}\,\mu(dz)+\lambda_s(y_s^{t,y})|x_s^{t,x;\xi,\rho}|^2 &\\
	 			-\int_{\mathcal{Z}}\frac{|u_s(y_s^{t,y})|^2|x_s^{t,x;\xi,\rho}|^2}{\gamma_s(y_s^{t,y},z) +u_s(y_s^{t,y})}\mu(dz) -\frac{|u_s(y_s^{t,y})|^2 |x_s^{t,x;\xi,\rho}|^2}{\eta_s(y_s^{t,y})}\bigg\}\,ds&
			\end{split}
			\\
		&\leq E^{\bar{\sF}_t} \int_t^\tau \left\{ \eta_s(y_s^{t,y})|\xi_s|^2 +\lambda_s(y_s^{t,y})\big| x_s^{t,x;\xi,\rho}  \big|^2 +\int_{\mathcal{Z}}\gamma_s(y_s^{t,y},z) |\rho_s(z)|^2\,\mu(dz)\right\}ds
	\end{aligned}
\end{multline}
for all $0\leq t\leq \tau <T$. In view of Lemma \ref{lem-admissible-control} we may with no loss of generality assume that process $x^{0,x;\xi,\rho}$ is monotone and hence,
\begin{align*} %\label{esti-001}
	\lim_{\tau\rightarrow T} E^{\bar\sF_t}\left\{ u_\tau(y_\tau^{t,y})\left|x_\tau^{t,x;\xi,\rho}\right|^2\right\} \leq\lim_{\tau\rightarrow T} \frac{c_1}{T-\tau} C (T-\tau)E^{\bar\sF_t}\int_\tau^T\left|\xi_s\right|^2ds = 0.
\end{align*}
Thus, taking the limit $\tau\rightarrow T$ in \eqref{xyzz} yields that $J_t(x,y;\xi,\rho)\leq u_t(y)x^2$ for any admissible control $(\xi,\rho)$. For $(\xi^*,\rho^*)$ we have equality in \eqref{xyzz}, which implies $u_t(y)x^2=J_t(x,y;\xi^*,\rho^*)$. But this in particular shows $\xi^*\in L^2_{\bar\sF}(0,T;\mathbb R)$, thus $(\xi^*,\rho^*)$ is admissible, attains the minimal cost, and hence is optimal. \hfill \qed

%Plugging the feedback control $(\xi^*,\rho^*)$ into (\ref{xyzz}) and letting $r$ tend to $T$,  we finally obtain that for any $t\in[0,T]$,
%%\[
%%	u_t(y_t^{0,y})|x_t^*|^2 = E^{\bar{\sF}_t} \int_t^T \left\{ \eta_s(y_s^{0,y})|\xi^*_s|^2 +\lambda_s(y_s^{0,y})| x_s^{*}|^2 +\int_{\mathcal{Z}}\gamma_s(y_s^{0,y},z) |\rho^*_s(z)|^2\,\mu(dz)   \right\}ds.
%%\]
%This shows that $\xi^*\in \cL^2_{\bar{\sF}}(0,T;\bR)$ and thus $(\xi^*,\rho^*)$ is admissible. Moreover, the minimum is indeed attained at $(\xi^*,\rho^*)$. 

\section{Existence of a solution to BSPDE \eqref{BSPDE-singlr}}

	As a result of the verification theorem there exists at most one solution $(u,\psi)$ to \eqref{BSPDE-singlr} that satisfies \eqref{ABC} and \eqref{eq-thm-u}. In this section, we prove existence of a solution with these properties. To this end, we set
\begin{equation} \label{hatF}
	\hat{F}(t,y,\phi(y)) := F(t,y,|\phi(y)|),\quad (t,y,\phi)\in \bR_+\times \bR^d\times L^{0}(\bR^d),
\end{equation}
and construct the solution as the limit of a sequence of such a solution to a family of BSPDEs with driver $\hat{F}$ and finite increasing terminal values. More precisely, for each $N\in\bN$, we consider the BSPDE
\begin{equation} \label{BSPDE-singlr-N}
  \left\{\begin{aligned}
  	-dv^N_t(y)& = \{\tilde{ \mathcal L} v^N_t(y) + \tilde{ \mathcal M} \zeta^N_t(y) +\theta(y) \hat{F}(t,y,(\theta^{-1} v^N_t)(y))\}\,dt-\zeta^N_t(y)\, dW_{t}, \quad (t,y)\in [0,T]\times \bR^d;\\
   	v^N_T(y) &=  N{\theta}(y),  \quad y\in\bR,
  \end{aligned}\right.
\end{equation}
that corresponds to the singular BSPDE \eqref{BSPDE-singlr-1}, with the pair $(F,\infty)$ being replaced by $(\hat{F},N {\theta})$.  We cannot appeal directly to Proposition \ref{prop-bspde-duqiutang} to prove existence of a solution to the preceding BSPDE, due to the quadratic dependence of the driver $\hat{F}$ on $|\phi(y)|$ in \eqref{hatF}. However, we expect $v^N$ to be finite and hence to be able to construct a solution by a standard truncation argument.

\begin{proposition} \label{prop-N}
  Let assumptions $(\mathcal{A}1)-(\mathcal{A}3)$ be satisfied.  For each $N\in\bN$, there exists a unique solution to (\ref{BSPDE-singlr-N}) such that 
\begin{equation*}
	(v^N,\zeta^N)\in (\mathcal{H}^1\cap \cL^{2,\infty}_{\sF}(0,T))\times \cL^2_{\sF}(0,T;H^{1,2}(\bR^d))
\end{equation*}
and $\theta^{-1} v^N\in \cL^{\infty}_{\sF}(0,T;L^{\infty}(\bR^d))$.

\end{proposition}

\begin{proof}
	For each $M\in\bN$ there exists a unique solution
\[
	(v^{N,M},\zeta^{N,M})\in (\mathcal{H}^1 \cap \cL^{2,\infty}_{\sF}(0,T))  \times \cL^2_{\sF}(0,T;H^{1,2}(\bR^d))
\]
to the BSPDE
\begin{equation} \label{BSPDE-singlr-N-M}
  \left\{\begin{aligned}
  	-dv_t^{N,M}(y) &= \begin{aligned}[t]&\bigg(\tilde{ \mathcal L} v^{N,M}_t + \tilde{ \mathcal M} \zeta_t^{N,M}+ \theta\lambda -\int_{\mathcal{Z}}\frac{ \theta^{-1} |{v}^{N,M}_t|^2}{\gamma_t(\cdot,z)+|\theta^{-1} {v}^{N,M}_t|}\,\mu(dz) \\&-\frac{(M\wedge|\theta^{-1}{v}^{N,M}_t|) |{v}^{N,M}_t|}{\eta_t}\bigg)(y)\,dt-{\zeta}^{N,M}(y)\, dW_{t}, \quad  (t,y)\in [0,T]\times \bR^d;\end{aligned}\\
    v^{N,M}_T(y)&= N\theta(y),  \quad y\in\bR^d,
    \end{aligned}\right.
\end{equation}
due to Proposition \ref{prop-bspde-duqiutang}. Putting
$$
\hat{v}_t(y)={\theta}(y)\left(N+\Lambda(T-t)\right),
$$
we verify that $(\hat{v},0)$ is a solution of the above BSPDE with $(\lambda,\gamma,M)$ being replaced by $(\Lambda,+\infty,0)$. The comparison principle stated in Corollary~\ref{cor-comprn-frm-DuQIUTang} yields 
$$
0\leq v^{N,M}_t(y)\leq \hat{v}_t(y),\quad \mathbb P\otimes dt \otimes dy\textrm{-a.e.},
$$
which implies for any $M\in\bN$ that
$$
0\leq \theta^{-1}(y) v^{N,M}_t(y)\leq N+\Lambda T,\quad \mathbb P\otimes dt \otimes dy\textrm{-a.e.}
$$
Hence, if $M>N+\Lambda T$, then $(v^{N,M},\zeta^{N,M})$ does not depend on $M$ and is in fact a solution to \eqref{BSPDE-singlr-N}. This also yields uniqueness of solutions as \eqref{BSPDE-singlr-N-M} admits a unique solution for each $M \in \mathbb{N}$. 
\end{proof}

	The proof of Proposition \ref{prop-N} shows that the solution $(v^N,\zeta^N)$ to \eqref{BSPDE-singlr-N} coincides with that of \eqref{BSPDE-singlr-N-M} for some $M\in\bN$. Hence, as an immediate consequence of Corollary~\ref{cor-comprn-frm-DuQIUTang} we obtain the following comparison principle. 

\begin{corollary}\label{cor-prop-N}
%   Therefore, by the comparison principle in Corollary \ref{cor-comprn-frm-DuQIUTang}, we further conclude that $v^N$ is increasing in $N$.
%  Moreover,
   Let assumptions $(\mathcal{A}1)-(\mathcal{A}3)$ be satisfied and let $(\bar{\lambda},\bar{\gamma},\bar{\eta})$ satisfy the same conditions as $(\lambda,\gamma,\eta)$. Suppose further that
\[
	(\bar{v},\bar{\zeta})\in \mathcal{H}^1\times \cL^2_{\sF}(0,T;H^{1,2}(\bR^d))
\]
with $\theta^{-1} \bar{v}\in \cL^{\infty}_{\sF}(0,T;L^{\infty}(\bR^d))$, is a solution to the following BSPDE:
\begin{equation}\label{BSPDE-singlr-bar}
  \left\{\begin{aligned}
  	-d\bar{v}_t(y) &=\{ \tilde{ \mathcal L} \bar v_t(y) + \tilde{ \mathcal M} \zeta^N_t(y) +{\theta}(y) \hat{F}(t,y,(\theta^{-1} v^N_t)(y))\}\,dt-\bar{\zeta}_t(y) \, dW_{t}, \quad (t,y)\in [0,T]\times \bR^d;\\
    \bar{v}_T(y)&=G(y),  \quad y\in\bR^d.
   \end{aligned}\right.
\end{equation}
If $(G,\bar{\lambda}, \bar{\gamma}, \bar{\eta}) \geq (N,\lambda, \gamma, \eta)$, respectively, $(G,\bar{\lambda}, \bar{\gamma}, \bar{\eta}) \leq (N,\lambda, \gamma, \eta)$, then for almost all $(\omega,y)$ it holds that  
\[
	\bar{v}_t(y) \geq v^N_t(y), \quad \mbox{respectively}, \quad \bar{v}_t(y) \leq v^N_t(y), \quad \forall t \in [0,T].
\]
%If $G\geq N{\theta}$, $\bar{\lambda}\geq \lambda$, $\bar{\gamma}\geq \gamma$ and $\bar{\eta}\geq(\leq)\eta$, then, for almost every $y\in\bR^d$, it holds almost surely that
%$$\bar{v}_t(y)\geq (\leq)  v^N_t(y),\quad \forall t\in[0,T].$$
\end{corollary}

	We are now ready to prove existence of a solution to our singular BSPDE that satisfies the assumptions of the verification theorem.

\begin{theorem} \label{thm-existence}
  Let assumptions $(\mathcal{A}1)-(\mathcal{A}3)$ be satisfied. Then the BSPDE \eqref{BSPDE-singlr} admits a solution $(u,\psi)$ satisfying \eqref{ABC} and \eqref{eq-thm-u}. %$({\theta}u,{\theta}\psi)1_{[0,t]}\in\mathcal{H}^1\times \cL^2_{\sF}(0,T;H^{1,2}(\bR^d))$ for any $t\in(0,T)$ and 
%\begin{equation} %\label{eq-thm-2-u}
%   \frac{c_0}{T-t}\leq u_t(y) \leq \frac{c_1}{T-t},\quad \mathbb{P}\otimes dt \otimes dy\text{-a.e.}, 
%\end{equation}
%with $c_0$ and $c_1$ being two  positive constants.
\end{theorem}
\begin{proof}
	%We first prove existence.
	By Proposition \ref{prop-N}, for each $N>2\Lambda+\kappa_0\mu(\mathcal{Z})$, there exists a unique solution $(v^N,\zeta^N)$ to (\ref{BSPDE-singlr-N}) such that $(v^N,\zeta^N)\in (  \mathcal{H}^1\cap \cL^{2,\infty}_{\sF}(0,T))\times \cL^2_{\sF}(0,T;H^{1,2}(\bR^d))$ and $\theta^{-1} v^N\in \cL^{\infty}_{\sF}(0,T;L^{\infty}(\bR^d))$. If one replaces the triple $(\lambda,\gamma,\eta)$ by $(\Lambda,+\infty,\Lambda)$ and $(0,0,\kappa_0)$, respectively, then a direct computation shows that respective solutions to~\eqref{BSPDE-singlr-N} are given by $(\bar{u}^N,0)$ and $(\tilde{u}^N,0)$, where 
\begin{align*}
	\bar{u}^N_t(y) & :=	\frac{\kappa_0\mu(\mathcal{Z}){\theta}(y)}{1-\frac{N}{N+\kappa_0\mu(\mathcal{Z})}e^{-\mu(\mathcal{Z})(T-t)}}-\kappa_0\mu(\mathcal{Z}){\theta}(y), \\
	\tilde{u}^N_t(y) & :=\frac{2\Lambda{\theta}(y)}{1-\frac{N-\Lambda}{N+\Lambda}\cdot e^{-2(T-t)}}-\Lambda {\theta}(y).
\end{align*}
From Corollary \ref{cor-prop-N}, we conclude that for almost every $y\in\bR^d$, it holds almost surely that
\begin{equation*}
	\bar{u}^N_t(y)\leq v^N_t(y)\leq \tilde{u}^N_t(y), \quad t \in [0,T).
\end{equation*}
 Denoting by $v$ the limit of the increasing sequence $\{v^N\}$, we deduce that for almost every $y\in\bR^d$ that almost surely
\begin{align}	\label{est-prf-thm-solvablty-ul}
	\frac{\kappa_0e^{-\mu(\mathcal{Z})T}{\theta}(y)}{T-t}\leq  v_t(y)\leq \frac{\Lambda {e}^{2T}{\theta}(y)}{T-t}, \quad  t\in [0,T).
\end{align}
Further, by dominated convergence,
\begin{equation*}  %\label{est-prf-thm-lmt}
 \lim_{N\rightarrow \infty}
 \|{\theta}(\cdot)\hat{F}(\cdot,\cdot,(\theta^{-1} v_{\cdot}^N)(\cdot))
 -{\theta}(\cdot)F(\cdot,\cdot,(\theta^{-1} v_{\cdot})(\cdot))  \|_{\cL^2_{\sF}(0,\tau;L^2(\bR^d))}=0, \quad   \tau\in(0,T). 
\end{equation*}

	We now use $v$ to construct the desired solution by analyzing a BSPDE on $[0,\tau]$ with terminal value $v_\tau$. More precisely, let us denote by
 $$(\bar{v},\zeta)\in \left(\cL^2_{\sF}(0,\tau;H^{1,2}(\bR^d))\cap \cS^2_{\sF}([0,\tau];L^{2}(\bR^d))\right)\times \cL^2_{\sF}(0,\tau;L^{2}(\bR^d))$$
 the unique solution for the following BSPDE (guaranteed by Proposition~\ref{prop-bspde-duqiutang} as $v_\tau \in \cL^2(\Omega,\sF_{\tau}; L^2(\mathbb{R}^d))$ by~\eqref{est-prf-thm-solvablty-ul}):
\begin{equation*}
  \left\{\begin{aligned}
  -d\bar{v}_t(y) &=\{\tilde{ \mathcal L} \bar v_t(y) + \tilde{ \mathcal M} \zeta_t(y) +{\theta}(y) \hat{F}(t,y,(\theta^{-1} v_t)(y))\}\,dt-\zeta_t(y)\,dW_t, \quad (t,y)\in [0,\tau)\times \bR^d;\\
  \bar{v}_{\tau}(y)&= v_{\tau}(y),  \quad y\in\bR^d.
    \end{aligned}\right.
\end{equation*}

We use this equation to show that $v$ lies in the right space. In view of estimate \eqref{est-prop-dqiutang} in Proposition \ref{prop-bspde-duqiutang}, we have as $N\rightarrow +\infty$,
\begin{multline*}
	\|(v^N-\bar{v})1_{[0,\tau]}\|_{\mathcal{H}^0} +\|\zeta^N-\zeta\|_{\cL^2_{\sF}(0,T;L^{2}(\bR^d))}\\
    \leq C\Big(\|v^N_{\tau}-v_{\tau}\|_{L^2(\Omega,\sF_{\tau};L^2(\bR^d))} +\|{\theta}\hat{F}(\cdot,\cdot,(\theta^{-1} v_{\cdot}^N)(\cdot)) -{\theta}F(\cdot,\cdot,(\theta^{-1} v_{\cdot})(\cdot))  \|_{\cL^2_{\sF}(0,\tau;L^2(\bR^d))}\Big) \longrightarrow 0.
  \end{multline*}
Thus,  
\[\bar{v}=v1_{[0,\tau]}\in  \mathcal{H}^0=\cL^2_{\sF}(0,\tau;H^{1,2}(\bR^d))\cap \cS^2_{\sF}([0,\tau];L^{2}(\bR^d)).
\]	
Hence, for each $\delta \in (0,\tau)$ there exists $\tilde\tau\in (\tau-\delta,\tau]$ such that $v_{\tilde\tau}\in L^2(\Omega,\sF_{\tilde\tau};H^{1,2}(\bR^d))$, and by Proposition \ref{prop-bspde-duqiutang}, we further have
\[
	(v1_{[0,\tilde\tau]},\zeta1_{[0,\tilde\tau]})\in (\cL^{2,\infty}_{\sF}(0,\tilde\tau)\cap\mathcal{H}^1) \times \cL^2_{\sF}(0,\tilde\tau;H^{1,2}(\bR^d)).
\]
This shows that $(u,\psi):=(\theta^{-1} v,\theta^{-1} \zeta)$
is a solution to BSPDE~\eqref{BSPDE-singlr} with the desired properties. 
\end{proof}

%%%%%%%%%%%%%%%%%%%%%%%%%%%%%%%
%%%%%%%%%%%%%%%%%%%%%%%%%%%%%%%
%%%%%%%%%%%%%%%%%%%%%%%%%%%%%%%

\section{Uniqueness and regularity}

	In this section we show that the solution to the BSPDE~(\ref{BSPDE-singlr}) constructed in the previous section is the unique non-negative solution to~(\ref{BSPDE-singlr}). Subsequently, using the existing $L^p$-theory of BSPDEs, we consider the regularity of the solution.

\subsection{Uniqueness}

	The following uniqueness result is based on the observation that any non-negative solution to~\eqref{BSPDE-singlr} automatically satisfies the growth condition of the verification theorem. 

\begin{theorem} \label{thm-minimal}
	Under assumptions $(\mathcal{A}1)-(\mathcal{A}3)$, the solution $(u,\psi)$ given in Theorem~\ref{thm-existence} is the unique non-negative solution to~\eqref{BSPDE-singlr} in the sense that if $(\bar{u},\bar{\psi})$ is another solution that satisfies \eqref{ABC} and $\bar{u}\geq 0$, $\mathbb{P}\otimes dt \otimes dx\text{-a.e.}$, then
\begin{equation*}
  \bar{u}_t(y)= u_t(y),\quad \mathbb{P}\otimes dt \otimes dy\text{-a.e.}
\end{equation*}
\end{theorem}

\begin{proof}
	In view of Theorem~\ref{thm-verification}, to establish the uniqueness statement it is sufficient to verify that $\bar u$ satisfies the growth condition \eqref{eq-thm-u}.

	Set $(\bar{v},\bar{\zeta})=({\theta}\bar{u},{\theta}\bar{\psi})$ and for $N\in\bN$, let $(v^N,\zeta^N)$ be the unique solution to~\eqref{BSPDE-singlr-N}. From the proof for Theorem \ref{thm-existence} we see that to establish the lower bound in~\eqref{eq-thm-u} one needs only to prove
\begin{equation} \label{eq-prf-thm-minimal}
  \bar{v}_t(y)\geq v^N_t(y),\quad \mathbb{P}\otimes dt \otimes dy\text{-a.e.}
\end{equation}
Putting $(\tilde{v},\tilde{\zeta})=(v^N-\bar{v},\zeta^N-\bar{\zeta})$ and noticing that for the moment one only has that $\eta^{-1}|\bar{v}|^2$ lies in $\mathcal{L}^1_{\sF}(0,t;L^1(\bR^d))$ instead of $\mathcal{L}^2_{\sF}(0,t;L^2(\bR^d))$, we apply the inequality for BSPDEs stated in Lemma \ref{lemma-appendix-ito} in the appendix. Since
\begin{equation*} %\label{eq-minimal-monotone}
  \left(F(t,y,(\theta^{-1}\phi_1)(y))- F(t,y,(\theta^{-1}\phi_2)(y))\right)(\phi_1-\phi_2)^+(y)\leq 0, \quad \mathbb{P}\otimes dt \otimes dy\text{-a.e.,}
\end{equation*}
for any pair of non-negative measurable functions  $\phi_1$ and $\phi_2 $ on $\bR^d$, and because $\sigma$ and $\bar{\sigma}$ are bounded and Lipschitz continuous, we obtain from that lemma for $\tau\in(0,T)$ and $t\in(0,\tau)$,
\begin{multline*}
	E\,\bigg\{ \|\tilde{v}^+_t\|^2_{L^2(\bR^d)}+\int_t^{\tau}\|\tilde{\zeta}_s1_{u> u_1}\|^2_{L^2(\bR^d)}\,ds \bigg\}\\
	\leq E\,\bigg\{ \|\tilde{v}^+_{\tau}\|^2_{L^2(\bR^d)}+\int_t^{\tau}2\langle \tilde{v}^+_s,\,a^{ij}_s\partial^2_{y_iy_j}\tilde{v}_s +\sigma^{jr}_s\partial_{y_j}\tilde{\zeta}^r_s +\tilde{b}^i_s\partial_{y^i} \tilde{v}_s +\beta_s^{\mathcal{T}}\tilde{\zeta}_s+c_s\tilde{v}_s \rangle\,ds\bigg\},
%	\end{aligned}
\end{multline*}
where the summation convention is applied. In view of assumptions $(\mathcal{A} 1-\mathcal{A} 3)$, using H\"{o}lder's inequality and the integration-by-parts formula, by adopting the ``standard machinery'' (see for instance \cite{QiuTangMPBSPDE11,QiuWei-RBSPDE-2013}) for linear equations, we  arrive at
\begin{multline*}
	E\,\bigg\{ \|\tilde{v}^+_t\|^2_{L^2(\bR^d)}+\int_t^{\tau}\|\tilde{\zeta}_s1_{u> u_1}\|^2_{L^2(\bR^d)}\,ds\bigg\}\\
	\leq E\bigg\{\|\tilde{v}^+_{\tau}\|^2_{L^2(\bR^d)} +\int_t^{\tau}\left\{C\|\tilde{v}^+_s\|_{L^2({\bR^d})}^2- \tfrac{1}{2}\kappa\|D \tilde{v}^+_s\|^2_{L^2(\bR^d)}+\|\tilde{ \zeta}_s1_{v^N>\bar{v}}\|^2_{L^2(\bR^d)}\right\}ds\bigg\}.
\end{multline*}
By Gronwall's inequality this implies
\[
E\|\tilde{v}^+_t\|^2_{L^2(\bR^d)}\leq C E \|\tilde{v}^+_{\tau}\|^2_{L^2(\bR^d)},
\]
where $C$ is independent of $\tau$ and $t$. As $\theta^{-1} v^N\in \cL^{\infty}_{\sF}(0,T;L^{\infty}(\bR^d))$ and $v^N\in\mathcal{H}^1$ by Proposition \ref{prop-N}, and
\[
	\tilde{v}^+=(v^N-\bar{v})^+\leq |v^N|,\quad  \mathbb{P}\otimes dt \otimes dy\text{-a.e.},
\]
we have by Fatou's lemma
\begin{align*}
	\int_{[0,T]\times\bR^d}E|\tilde{v}^+_t(y)|^2\,dydt\leq CT \limsup_{\tau\uparrow T} \int_{\bR^d}E |\tilde{v}^+_{\tau}(y)|^2  \,dy\leq CT \int_{\bR^d}E\limsup_{\tau\uparrow T} |\tilde{v}^+_{\tau}(y)|^2 \,dy=0.
\end{align*}
Hence, the lower bound of \eqref{eq-thm-u} holds for $\overline u$.

	To establish the upper bound in \eqref{eq-thm-u} we extend an argument given in \cite{GraeweHorstSere13} and consider the deterministic function
\[ 
	\hat u_t:=\Lambda\coth(T-t)=\frac{2\Lambda}{1-e^{-2(T-t)}}-\Lambda\leq \frac{\Lambda e^{2T}}{T-t}.
\]
Then, $(\hat u, 0)$ is a solution to \eqref{BSPDE-singlr} with the triple $(\lambda,\gamma,\eta)$ being replaced by $(\Lambda,+\infty,\Lambda)$. Moreover, $(\hat u, 0)$ remains a solution  when shifted in time, i.e., for $\delta\in[0,T)$ the pair $(\hat u_{\,\cdot\,+\delta} ,0)$ is the solution to \eqref{BSPDE-singlr} associated with $(\Lambda,+\infty,\Lambda)$, but with a singularity at $t=T-\delta$. Hence, noting that
\[\left(F(t,y,(\theta^{-1}\phi_1)(y))-\Lambda+\Lambda^{-1}|(\theta^{-1}\phi_2)(y)|^2\right)(\phi_1-\phi_2)^+(y)\leq 0, \quad \mathbb{P}\otimes dt \otimes dy\text{-a.e.,}\]
for any pair of non-negative measurable functions  $\phi_1$ and $\phi_2 $ on $\bR^d$, using arguments similar to those used in the first part of this proof, we conclude
\begin{equation*}
	\int_{[0,T-\delta]\times\mathbb R^d}E|(\theta\bar u_t-\theta\hat u_{t+\delta})^+(y)|^2\,dydt\leq C(T-\delta)\int_{\mathbb R^d} E\limsup_{\tau\uparrow T-\delta}|(\theta\bar u_{\tau}-\theta\hat u_{\tau+\delta})^+(y)|^2\,dy=0.
\end{equation*}
This yields, 
\[
	\bar u_t(y)\leq\frac{\Lambda e^{2T}}{T-\delta-t}, \quad \mathbb{P}\otimes dt \otimes dy\text{-a.e.}
\]
Finally, letting $\delta\rightarrow 0$ we obtain the desired upper bound.
\end{proof}

\subsection{Regularity}

	We proved so far that, under assumptions $(\mathcal{A}1)-(\mathcal{A}3)$, the BSPDE~\eqref{BSPDE-singlr} admits a unique non-negative solution $(u,\psi)$ that satisfies~\eqref{ABC}. This solution automatically satisfies the growth condition \eqref{eq-thm-u} and $V(t,y,x):=u_t(y)x^2$ coincides with the value function of \eqref{value-func} for almost every $y\in\bR^d$. 

	{Inspired by the $L^p$-theory ($p>2$) of BSPDEs, we now prove additional regularity properties of $u$ under the following additional assumption:
\begin{itemize}
  \item[$(\mathcal{A}4)$] $\sigma$ is spatially invariant (does not depend on $y$).
\end{itemize}}

\begin{theorem}
	Under assumptions $(\mathcal{A}1)-(\mathcal{A}4)$, let $(u,\psi)$ be the unique non-negative solution to~\eqref{BSPDE-singlr} that satisfies~\eqref{ABC}. Then, for any $p\in(2,+\infty)$,
\[
\theta(\cdot) u_{\cdot}\Big(\cdot+\int_0^{\cdot}\sigma_s\,dW_s\Big)
\in\bigcap\limits_{\tau\in(0,T)}\bigcap\limits_{\delta\in(0,1)}\cL^{2,\infty}_{\sF}(0,\tau)\cap  \cS^p_{\sF}([0,\tau];C^{\delta}(\bR^d)).
\]
Furthermore, the function $V(t,y,x):=u_t(y)x^2$ coincides with the value function of \eqref{value-func} for every $y\in\bR^d$. % and the optimal (feedback) control is given by~\eqref{eq-thm-control}.
\end{theorem}

\begin{proof}
For each $N\in\bN$, let $(v^N,\zeta^N)$ be the unique solution to the BSPDE \eqref{BSPDE-singlr-N}. Our goal is to derive additional regularity properties under $(\mathcal{A}4)$ using the $L^p$-theory for BSPDEs developed in \cite{DuQiuTang10}. 

The results of \cite{DuQiuTang10} do not allow the linear term $\beta^{\mathcal T}\zeta^N$ in the drift part of the BSPDE, though. To overcome this problem,  we make the following change of variables:
{\allowdisplaybreaks
\begin{align*}
  y_t^y &:=  y+\int_0^t\sigma_s\,dW_s, \quad (t,y)\in[0,T]\times\bR^d;\\
  \bar{a}_s(y) &:=  \frac{1}{2}\bar{\sigma}_s(y)\bar\sigma_s^{\mathcal{T}}(y), \quad y\in\bR^d;\\
	(\tilde{u}^N_t,\tilde{\psi}^N_t)(y)&:=(\theta^{-1} v^N_t,\theta^{-1} \zeta^N_t+\sigma^{\mathcal{T}}_tD(\theta^{-1} v^N_t))(y_t^y), \quad (t,y)\in[0,T]\times\bR^d;\\
	(\tilde{v}^N_t,\tilde{\zeta}^N_t)(y) &:=({\theta}\tilde{u}^N_t,{\theta}\tilde{\psi}^N_t)(y), \quad (t,y)\in[0,T]\times\bR^d.
 \end{align*}}%
%  and for $(t,y)\in[0,T]\times\bR^d$,
%  \begin{equation}\label{eq-transfm}
%  (\tilde{u}^N_t,\tilde{\psi}^N_t)(y)=(\theta^{-1} v^N_t,\theta^{-1} \zeta^N_t+\sigma^{\mathcal{T}}_tD(\theta^{-1} v^N_t))(y_t^y)\quad \textrm{ and } \quad
%  (\tilde{v}^N_t,\tilde{\zeta}^N_t)(y)=({\theta}\tilde{u}^N_t,{\theta}\tilde{\psi}^N_t)(y).
%  \end{equation}
%
%Putting
%\[
%  \bar{a}_s(y):=\frac{1}{2}\bar{\sigma}_s(y)\bar\sigma_s^{\mathcal{T}}(y),
%\]
Then, applying the It\^o-Wentzell formula for distribution-valued processes (see \cite[Theorem~1]{Krylov_09}), we have almost surely that
\begin{multline} \label{eq1-sec7}
		\tilde{v}^N_t(y) =  N{\theta}(y)+ \int_t^{T} \Big\{\tr\left( \bar{a}_s(y_s^y)D^2 \tilde{v}^N_s(y)\right)+ \bar{b}^{\mathcal{T}}_s(y)D\tilde{v}^N_s(y)+\bar{c}_s(y)\tilde{v}^N_s(y)
		+{\theta}(y)\hat{F}(s,y_s^y,(\theta^{-1}\tilde{v}^N_s)(y))\Big\}\,ds\\-\int_t^{T}\tilde{\zeta}^N_s(y)\, dW_{s}, \quad dy\text{-a.e.}  \quad \forall t \in[0,T]
\end{multline}
with
\begin{align*}
	\bar{b}_t^i(y) &:= b^i_t(y_t^y)+\frac{4q}{1+|y|^2}\sum_{j=1}^d a^{ij}_t(y_t^y)y^j, \quad i=1,\ldots,d;\\
   \bar{c}_t(y) &:=\frac{2q}{1+|y|^2}\bigg(\tr(a_t(y_t^y))+\sum_{i=1}^dy^ib^i_t(y_t^y)+\frac{2(q-1)}{1+|y|^2}\sum_{i,j=1}^da_t^{ij}(y_t^y)y^iy^j\bigg).
\end{align*}
From this representation we see that we also have a BSDE representation of $(\tilde{v}^N, \tilde \zeta^N)$ from which we will obtain strong regularity properties. Specifically, by Proposition~\ref{prop-bspde-duqiutang}, there exists a unique solution
  $$(\bar{v}^N,\bar{\zeta}^N)\in \left( \mathcal{H}^1\cap \cL^{2,\infty}_{\sF}(0,T) \right)\times \cL^2_{\sF}(0,T;H^{1,2}(\bR^d)),$$
to the BSPDE
\begin{equation} \label{BSPDE-sec7}
  \left\{\begin{aligned}
  		-d\bar{v}^N_t(y) &= \begin{aligned}[t]&\Big\{\tr\left( \bar{a}_t(y_t^y) D^2 \bar{v}^N_t(y)\right)+\bar{b}^{\mathcal{T}}_t(y)D \bar{v}^N_t(y)+\bar{c}_t(y)\bar{v}^N_t(y)+{\theta}(y)\bar{\lambda}_t(y_t^y)-\frac{\left|\tilde{v}^N_t(y)\bar{v}^N_t(y)\right|}{\theta(y)\bar{\eta}_t(y_t^y)}\\&-\int_{\mathcal{Z}}\frac{\theta^{-1}(y) |\bar{v}^N_t(y)|^2}{\bar{\gamma}_t(y_t^y,z) +|\theta^{-1}(y)\bar{v}^N_t(y)|}\mu(dz)\Big\}\,dt -\bar{\zeta}^N_t(y)\, dW_{t}, \quad (t,y)\in [0,T]\times \bR^d;\end{aligned}\\
    	\bar{v}^N_T(y)&= N{\theta}(y),  \quad y\in\bR^d.
    \end{aligned}\right.
\end{equation}
By definition, the solution satisfies \eqref{eq1-sec7}. As $\theta^{-1} \tilde{v}^N\in \cL^{\infty}_{\sF}(0,T;L^{\infty}(\bR^d))$ we can use the comparison principle stated in Corollary \ref{cor-comprn-frm-DuQIUTang} to deduce (similarly to the proof of Proposition \ref{prop-N}) that $\theta^{-1} \bar{v}^N\in \cL^{\infty}_{\sF}(0,T;L^{\infty}(\bR^d))$. Hence, by \cite[Proposition 6.4]{DuQiuTang10}, we further have 
\[
	\bar{v}^N\in \cS^p_{\sF}([0,T];H^{1,p}(\bR^d))\cap \cL^p_{\sF}(0,T;H^{2,p}(\bR^d)) \quad \text{for any }
	p\in(2,+\infty).
\]	
Thus, by Sobolev embedding theorem, $\bar{v}^N\in \cS^p_{\sF}([0,T];C^{\delta}(\bR^d))$, for any $\delta\in(0,1)$. Therefore, $\bar{v}^N_t(y)$ is almost surely continuous in $(t,y)\in[0,T]\times\bR^d$.

	Next, we are going to show that 
\[
	\tilde{v}^N_t(y)=\bar{v}^N_t(y),\quad \mathbb{P}\otimes dy\textrm{-a.e.}
\]
To this end, we show that both $(\tilde{v}^N,\tilde{\zeta}^N)$ and $(\bar{v}^N,\bar{\zeta}^N)$ satisfy the same BSDE. Specifically, let 
\begin{equation*} %\label{S-proce-tild-y}
  \tilde{y}_t^{s,y} :=y+\int_s^t\bar{b}_r(\tilde{y}^{s,y}_r)\,dr +\int_s^t\bar{\sigma}_r(y_r^{\tilde{y}_r^{s,y}})\,dB_r, \quad 0\leq s\leq t\leq T.
\end{equation*}
Since $(\tilde{v}^N_t,\ \tilde{\zeta}_t^y)(y)={\theta}(y)(\theta^{-1} v^N_t,\ \theta^{-1}\zeta^N_t+\sigma^{\mathcal{T}}_tD(\theta^{-1} v^N_t))(y_t^{y})$, one checks through standard but tedious computations that both $\bar{v}^N$ and $\tilde{v}^N$ are bounded and satisfy the following BSDE:
\begin{multline*} %\label{eq-bsde-sec7}
	\check{v}^N_t(\tilde{y}_t^{s,y})  =  N{\theta}(\tilde{y}_T^{s,y}) +\int_t^T\bigg\{\bar{c}_r(\tilde{y}_r^{s,y})\check{v}^N_r(\tilde{y}_r^{s,y}) +{\theta}(\tilde{y}_r^{s,y})\bar{\lambda}_t(y_r^{\tilde{y}_r^{s,y}}) -\frac{\theta^{-1}(\tilde{y}_r^{s,y}) \left|\check{v}_t(\tilde{y}_r^{s,y})\tilde{v}_t(\tilde{y}_r^{s,y})\right|}{\bar{\eta}_t(y_r^{\tilde{y}_r^{s,y}})}
	\\-\int_{\mathcal{Z}}\frac{\theta^{-1}(\tilde{y}_r^{s,y}) |\check{v}_t(\tilde{y}_r^{s,y})|^2}{\bar{\gamma}_t(y_r^{\tilde{y}_r^{s,y}},z)+\theta^{-1}(\tilde{y}_r^{s,y}) |\check{v}_t(\tilde{y}_r^{s,y})|}\mu(dz)\bigg\}\,dr -\int_t^T\check{\zeta}^N_r(\tilde{y}_r^{s,x})\,dW_r- \int_t^T\bar{\sigma}^{\mathcal{T}}_r(\tilde{y}_r^{s,y})D\check{v}^N_r(\tilde{y}_r^{s,x})\,dB_r.
\end{multline*}
This BSDE has a unique solution. In view of Lemma \ref{lem-Yang-Tang-13} and Corollary \ref{cor-ito-wentzell}, we therefore conclude
   $$\tilde{v}^N_t(\tilde{y}_t^{s,y})=\bar{v}^N_t(\tilde{y}_t^{s,y}),\quad \mathbb{P}\otimes dy\text{-a.e.} \quad \forall \, 0\leq s\leq t\leq T,$$
   where we note that both $\tilde{v}^N$ and $\bar{v}^N$ belong to $\mathcal{H}^1\cap \cL^{2,\infty}_{\sF}(0,T)$.   Taking $s=t$, we have
$$\tilde{v}^N_t(y)=\bar{v}^N_t(y),\quad \mathbb{P}\otimes dy\text{-a.e.} \quad \forall t\in[0,T].$$
Since the BSPDE \eqref{BSPDE-sec7} has a unique solution we also obtain
\[
	(\tilde{v}^N,\tilde{\zeta}^N)=(\bar{v}^N,\bar{\zeta}^N)\quad \textrm{in } \mathcal{H}^1\times \cL^2_{\sF}(0,T;H^{1,2}(\bR^d)).
\]
The regularity properties of $\bar{v}^N$  imply that $\tilde{u}^N_t(y)$, $v^N_t(y)$ and $\tilde{v}^N_t(y)$ are all continuous in $(t,y)$ with probability~$1$. In view of the proof of Theorem \ref{thm-existence}, we have $\{\tilde{v}^N_t(y)\}$ converges increasingly to $\theta(y)\theta^{-1}{v}_t(y_t^y)$ for every $(t,y)\in[0,T]\times\bR^d$ with probability~$1$, as $N$ goes to infinity. Setting \[
(\tilde{v}_t(y),\tilde{\zeta}_t(y)):=\theta(y)((\theta^{-1}{v}_t)(y_t^y),(\theta^{-1}{\zeta}_t)(y_t^y)),
\] 
we obtain $(\tilde{v},\tilde{\zeta})1_{[0,\tau]}
\in \left(\mathcal{H}^1\cap \cL^p_{\sF}(0,T;H^{2,p}(\bR^d))\right)
\times \cL^2_{\sF}(0,T;H^{1,2}(\bR^d))$ for all $\tau\in(0,T)$ and $p\in (2,\infty)$, and
\[
	\frac{c_0{\theta}(y)}{T-t}\leq \tilde{v}_t(y) \leq \frac{c_1{\theta}(y)}{T-t},\quad \mathbb{P}\otimes dt \otimes dy\text{-a.e.}
\]
Moreover, for every $\tau\in(0,T)$, it holds almost surely 
\begin{multline*}
	\tilde{v}_t(y) = \tilde{v}_{\tau}(y) +\int_t^{\tau}\big\{\tr\left( \bar{a}_s(y_s^y) D^2 \tilde{v}_s(y)\right) +\bar{b}^{\mathcal{T}}_s(y)D\tilde{v}_s(y)+\bar{c}_s(y)\tilde{v}_s(y) +{\theta}(y)F(s,y_s^y,(\theta^{-1}\tilde{v}_s)(y))\big\}\,ds\\
	-\int_t^{\tau}\tilde{\zeta}_s(y)\, dW_{s}, \quad dy\text{-a.e.}
\end{multline*}
Again, by \cite[Propostion 6.4]{DuQiuTang10}, we further have \[\tilde{v}\in \cS^p_{\sF}([0,\tau];H^{1,p}(\bR^d))\cap \cL^p_{\sF}(0,\tau;H^{2,p}(\bR^d)), \quad p\in(2,+\infty),\] and thus, by Sobolev embedding theorem, $\tilde{v}\in \cS^p_{\sF}([0,\tau];C^{\delta}(\bR^d))$ for every $\delta\in(0,1)$. Therefore, both $\tilde{v}_t(y)$ and $u_t(y)=\theta^{-1}(y-\int_0^t\sigma_sdW_s)\tilde{v}_t(y-\int_0^t\sigma_sdW_s)$ are almost surely continuous in $(t,y)\in[0,\tau]\times\bR^d$. Hence,
\[
	V(t,y,x):=u_t(y)x^2,\quad (t,x,y)\in[0,T]\times \bR\times \bR^d,
\]
coincides with the value function of \eqref{value-func} for every $y\in\bR^d$.
\end{proof}

%%%%%%%%%%%%%%%%%%%%%%%%%%%%%%%%%%%%%%
%%%%%%%%%%%%%%%%%%%%%%%%%%%%%%%%%%%%%%
%%%%%%%%%%%%%%%%%%%%%%%%%%%%%%%%%%%%%%

\begin{appendix}

\section{Three results on BSPDEs}
%
%For the reader's convenience this appendix recalls three results on BSPDEs which are used throughout this paper.
%
%\subsection{An existence and uniqueness result for BSPDE}
%
%The following existence and uniqueness of solutions result for BSPDEs is established in \cite[Theorem 5.5]{DuQiuTang10}.

\begin{proposition}[{\cite[Theorem 5.5]{DuQiuTang10}}]  \label{prop-bspde-duqiutang}
Let the coefficients $b$, $\sigma$ and $\bar{\sigma}$ satisfy the assumptions $(\mathcal{A}1)-(\mathcal{A}3)$. Suppose that the random function
  $f(\cdot,\cdot,\cdot,\vartheta,y,z)\in \cL^2_{\sF}(0,T;L^2(\bR^d))$
 for any $(\vartheta,y,z)\in \bR\times\bR^{d}\times\bR^{ m}$ and that there exists a positive constant $L_0$ such that for all $(\vartheta_1,y_1,z_1),(\vartheta_2,y_2,z_2)\in \bR\times\bR^d\times\bR^{ m}$
   and $(\omega,t,x)\in \Omega\times[0,T]\times\bR^d$,
   \begin{equation*}
     \begin{aligned}
       |f(\omega,t,x,\vartheta_1,y_1,z_1)-f(\omega,t,x,\vartheta_2,y_2,z_2)|\leq\,& L_0(|\vartheta_1-\vartheta_2|+|y_1-y_2|+|z_1-z_2|).
     \end{aligned}
   \end{equation*}
Then, for any given $G\in L^2(\Omega,\sF_T;H^{k,2}(\bR^d))$ with $k\in\{0,1\}$, the BSPDE
\begin{equation} \label{BSPDE-dqiutang}
  \left\{\begin{aligned}
  		-du_t(x) & =\begin{aligned}[t]&\{\tr\left(a_t(x)D^2u_t(x)+ D\psi_t(x)\sigma^{\mathcal{T}}_t(x)\right)+ b^{\mathcal{T}}_t(x)D u_t(x)\\&+f(t,x,x_t(x),D  u_t(x),\psi_t(x))\}\,dt -\psi_t(x)\, dW_{t}, \quad (t,x)\in [0,T]\times \bR^d;\end{aligned}\\
    	u_T(x)& = G(x),  \quad x\in\bR^d,
    \end{aligned}\right.
\end{equation}
admits a unique solution $(u,\psi)\in \mathcal{H}^k\times \cL^2_{\sF}(0,T;H^{k,2}(\bR^d))$, i.e., it holds almost surely that
\begin{multline} \label{eq-defn-solution-A1}
		\left\langle\varphi,\,u_t\right\rangle= \langle\varphi,\,u_{T}\rangle+\int_t^{T}\Big\{\langle \varphi,\, \tr\left( a_sD^2 u_s + D \psi_s\sigma^{\mathcal{T}}_s\right)+b^{\mathcal{T}}_sD u_s + f(s,u_s,D  u_s,\psi_s)\rangle\Big\}\,ds \\-\int_t^{T}\left\langle\varphi,\,\psi_sdW_s\right\rangle \quad \forall\varphi\in C_c^{\infty}(\bR^d), \quad t\in[0,T],
\end{multline}
where $C_c^{\infty}(\bR^d)$ is the set of all the infinitely differentiable functions with compact supports on $\bR^d$. Moreover, $u\in\cL^{2,\infty}_{\sF}(0,T)$ if $k=1$, and there exists a constant $C$ that depends only on $\kappa$, $L$, $L_0$, $\Lambda$ and~$T$ such that
\begin{multline} \label{est-prop-dqiutang}
	\|u\|_{\mathcal{H}^k}+\|u\|_{\cL^{2,\infty}_{\sF}(0,T)}1_{k=1} +\|\psi\|_{\cL^2_{\sF}(0,T;H^{k,2}(\bR^d))}\\
	\leq C\left(\|f(\cdot,\cdot,\cdot,0,0,0)\|_{\cL^2_{\sF}(0,T;L^2(\bR^d))} +\|G\|_{L^2(\Omega,\sF_T;H^{k,2}(\bR^d))}\right).
\end{multline}
%  for a constant $C$ that depends only on $\kappa$, $L$, $L_0$, $\Lambda$ and $T$.
\end{proposition}

%\subsection{A comparison principle}
By using the standard denseness arguments, one can easily check that for $k=1$, the requirement by \eqref{eq-defn-solution-A1} with test functions for the definition of solution is equivalent to the corresponding one holding almost everywhere in Definition \ref{defn-solution}. The nonlinear term $f$ in Proposition \ref{prop-bspde-duqiutang}  can be rewritten in linear form as
\begin{equation} \label{eq-linearization}
	f(t,x,\vartheta,y,z)=\alpha \vartheta + \beta^{\mathcal T} y+\vartheta^{\mathcal T}z +f(t,x,0,0,0), \quad (\vartheta,y,z)\in \bR\times\bR^d\times\bR^m,
\end{equation}
where
\begin{align*}
   \alpha&=\frac{f(t,x,\vartheta,y,z)-f(t,x,0,y,z)}{\vartheta} 1_{\vartheta\neq 0};\\
   \beta^i&=\frac{f(t,x,0,y^{(i)},z)-f(t,x,0,y^{(i-1)},z)}{y_i} 1_{y_i\neq 0},\quad i=1,\ldots,d;\\
   \vartheta^k&=\frac{f(t,x,0,0,z^{(k)})-f(t,x,0,0,z^{(k-1)})}{z_k} 1_{z_k\neq 0},\quad k=1,\ldots,m;\\
   y^{(i)}&=(y_1,\ldots,y_i,0,\ldots,0),\quad y^{(0)}=0\in\bR^d,\quad i=1,\ldots,d;\\
   z^{(k)}&=(z_1,\ldots,z_k,0,\ldots,0),\quad z^{(0)}=0\in\bR^m,\quad k=1,\ldots,m.
\end{align*}
Thus, the comparison principle for linear BSPDEs \cite[Theorem 6.3]{DuQiuTang10} implies immediately the following result.

\begin{corollary}[{Corollary of \cite[Theorem 6.3]{DuQiuTang10}}]\label{cor-comprn-frm-DuQIUTang} 
  Under the hypothesis of Proposition~\ref{prop-bspde-duqiutang}, for $k=1$, suppose the pair $(G',f')$ satisfies the same conditions as $(G,f)$ in Proposition \ref{prop-bspde-duqiutang}. Let $(u,v)$ and $(u',v')$ be the respective solutions to the BSPDE \eqref{BSPDE-dqiutang} and  assume furthermore that for almost every $(\omega,t,x)\in \Omega\times[0,T]\times\bR^d$ it holds
\[
	f(\omega,t,x,u_t,D  u_t,v)\geq f'(\omega,t,x,u_t,D  u_t,v)\quad \textrm{and} \quad G(\omega,x) \geq G'(\omega,x).
\]
Then, $u\geq u'$,  $\mathbb{P}\otimes dt \otimes dx$-a.e.
\end{corollary}

The corollary can be verified by applying the linearization \eqref{eq-linearization} to the function
$$
\tilde{f}(t,x,\vartheta,y,z):=f'(\omega,t,x,u_t,D  u_t,v)-f'(t,x,u_t+\vartheta,D  u_t+y,v+z).
$$
The proof is standard and hence omitted. We close this appendix with the following lemma on an inequality for the positive part of the solutions to BSPDEs, whose proof will be sketched below.

\begin{lemma} \label{lemma-appendix-ito}
  Let $u\in \mathcal{H}^0$. Suppose that for any $\varphi\in C_c^{\infty}(\bR^d)$, almost surely
\begin{equation*} %\label{eq bspde in ito formula}
	\langle \varphi,\,u_t\rangle=\langle \varphi,\,G\rangle+\int_t^T\bigg\{\langle\varphi,\, h_s+f_s\rangle-\sum_{i=1}^d\langle\partial_{x^i}\varphi,\,g_s^i\rangle\bigg\}\,ds-\int_t^T\langle \varphi,\,\zeta_s\, dW_s\rangle, \quad t\in[0,T],
\end{equation*}
where $G\in L^2(\Omega,\sF_T,L^2(\bR^d))$;  $\zeta,f,g\in \cL^{2}_{\sF}(0,T;L^2(\bR^d))$ and $h\in \cL^1_{\sF}(0,T;L^1(\bR^d))$. Moreover, assume $ h_s(x)u^+_s(x) \leq 0$, $\mathbb{P}\otimes dt\otimes dx$-a.e. Then, it holds almost surely that
\begin{multline}\label{inequ-ito}
  \|u^+_t\|^2_{L^2(\bR^d)}+\int_t^{T}\|\zeta_s1_{u> 0}\|^2_{L^2(\bR^d)}\,ds
  \\\leq\|G^+\|^2_{L^2(\bR^d)}+ 2\int_t^{T}\bigg\{\langle u^+_s,\,f_s \rangle-\sum_{i=1}^d\langle \partial_{x^i}u_s^+,\,g_s^i \rangle\bigg\}\,ds - 2\int_t^{T}\langle u^+_s,\,\zeta_s\,dW_s\rangle, \quad t\in[0,T].
\end{multline}
\end{lemma}

{\em Sketch of the proof.} 
	The pair $(u,\zeta)$ is the unique solution in $\mathcal{H}^0\times \cL^{2}_{\sF}(0,T;L^2(\bR^d))$ to the linear BSPDE
\begin{equation*} 
  \left\{\begin{aligned}
  		-d{u}_t(x) & =\bigg\{ \Delta u_t(x)
  		 +f_t+h_t+\sum_{i=1}^d\partial_{x^i}(g^i_t-\partial_{x^i}u_t(x))\bigg\}\,dt -\zeta_t(x)\, dW_{t}, \quad (t,x)\in [0,T]\times \bR^d;\\
    	{u}_T(x)& = G(x),  \quad x\in\bR^d.
    \end{aligned}\right.
\end{equation*}
If $h\in \cL^{2}_{\sF}(0,T;L^2(\bR^d))$, then \eqref{inequ-ito} follows from \cite[Corollary 3.11]{QiuWei-RBSPDE-2013}. For $h\in \cL^1_{\sF}(0,T;L^1(\bR^d))$, it can be verified using a standard approximation method. To this end, we first observe that the proof of \cite[Proposition~2]{DenisMatoussi2009} of the It\^o formula for \textit{forward} SPDEs is independent of the boundedness of the domain~$\mathcal{O}$ therein and hence the result extends to $\mathcal{O}=\bR^d$. Thus, for any function $\Phi:\bR\rightarrow\bR$ with bounded derivatives $\Phi'$ and $\Phi''$ and $\Phi^\prime(0)=0$, it holds almost surely that
\begin{multline}\label{eq-ito-appendx}
	\int_{\bR^d}\Phi(u_t(x))\,dx+\frac{1}{2}\sum_{r=1}^m\int_t^T\langle \Phi''(u_s)\zeta^r_s,\, \zeta^r_s\rangle \,ds+\int_t^T\langle \Phi'(u_s),\,\zeta_s\, dW_s\rangle\\
	=\int_{\bR^d}\Phi(G(x))\,dx+\int_t^T\bigg\{\langle \Phi'(u_s),\, f_s+h_s\rangle  -\sum_{i=1}^d\langle\Phi^{\prime\prime}(u_s)\partial_{x^i}u_s,\,g_s^i\rangle\bigg\}\,ds, \quad t\in[0,T].
\end{multline}
If $\Phi'(y)=\Phi'(y)1_{(0,\infty)}(y)\geq 0$, then our assumptions on $h$ yield almost surely that
\begin{equation}\label{eq-ito-incr}
	\text{LHS of \eqref{eq-ito-appendx}}\leq\int_{\bR^d}\Phi(G(x))\,dx +\int_t^T\bigg\{\langle \Phi'(u_s),\, f_s\rangle\,ds -\sum_{i=1}^d\langle \Phi''(u_s)\partial_{x^i} u_s,\, g^{i}_s\rangle\bigg\}\,ds, \quad t\in[0,T].
\end{equation}
We can generalize the above inequality to  $\Phi'$ being unbounded, %increasing and twice continuously differentiable functions $\Phi$ with bounded second derivative, 
by approximating $\Phi$ and passing to the limit in \eqref{eq-ito-incr}. Then it remains to apply inequality \eqref{eq-ito-incr} to the function $\Psi:y\mapsto (y^+)^2$. Though $\Psi$ is not regular enough, this can be done using the same approximation method as in Step 2 of the proof of \cite[Lemma~3.5]{QiuTangMPBSPDE11}. \hfill \qed

\section{Proof of Corollary \ref{cor-ito-wentzell}}
  In Lemma~\ref{lem-Yang-Tang-13}, $\Phi$ can be seen as an $L^2({\bR^d})$-valued continuous semi-martingale. Thus, $\Phi\in \mathcal{H}^0$ and we can further verify that 
\[
	(\Phi,\Upsilon)\in \left(\mathcal{H}^0\cap \cL^2_{\sF}(0,T;H^{2,2}(\bR^d))\right) \times\cL^2_{\sF}(0,T;H^{1,2}(\bR^d))
\] 
satisfies \eqref{BSPDE-dqiutang} with
\[
	f(t,y):=F_t(y)-\left(\tr\left(a_t(y)D^2 \Phi_t(y) + D\Upsilon_t(y)\sigma^{\mathcal{T}}_t(y)\right) +b^{\mathcal{T}}_t(y)D \Phi_t(y)\right).
\]
Thus, $\Phi\in\cL^{2,\infty}_{\sF}(0,T)\cap \mathcal{H}^1$ by Proposition \ref{prop-bspde-duqiutang}. For each $N\in\bN$,  let $(u^N,\psi^N)\in \mathcal{H}^1\times\cL^2_{\sF}(0,T;H^{1,2}(\bR^d))$  be the unique solution to 
\begin{equation}\label{BSPDE-cor-1}
  \left\{\begin{aligned}
  		-du^N_t(y) &=\begin{aligned}[t]&\big\{\tr\left( a_t(y)D^2 u^N_t(y)+ D \psi^N_t(y)\sigma^{\mathcal{T}}_t(y)\right)+b^{\mathcal{T}}_t(y)D u^N_t(y) +N\wedge| F_t(y)| \big\}\,dt\\ &-\psi^N_t(y)\, dW_{t}, \quad (t,y)\in [0,T]\times \bR^d;\end{aligned}\\
    u^N_T(y)&= N\wedge |G(y)|,  \quad y\in\bR^d.
    \end{aligned}\right.
\end{equation}
%({\bf{In view of Remark \ref{rmk-defn-solution}, I do not see the connection with that remark}
%Answer: This remark was used to state the equivalence relationships between different definitinos of strong solutions for BSPDE}
%%),
By Lemma \ref{lem-Yang-Tang-13}, we have for almost every $y\in\bR^d$,
\begin{multline*}
    u^N_t(y_t^{s,y})
    = N\wedge |G(y_T^{s,y})| +\int_t^T N\wedge| F_r(y_r^{s,y})|\,dr -\int_t^T\Big\{\sigma_r^{\mathcal{T}}(y_r^{s,y})D  u^N_r(y_r^{s,y})+\psi^N_r(y_r^{s,y})\Big\}\,dW_r
     \\ - \int_t^T\bar{\sigma}_r^{\mathcal{T}}(y_r^{s,y})D  u^N_r(y_r^{s,y})\,dB_r,
\end{multline*}
where all the compositions are well defined under the measure $\mathbb{P}\otimes dt \otimes dy$. In particular,
\[
  u^N_s(y)=E^{\sF_s}\bigg\{ N\wedge |G(y_T^{s,y})|+ \int_s^TN\wedge |F_r(y_r^{s,y})|\,dr\bigg\},
\]
while Proposition \ref{prop-bspde-duqiutang} yields a constant $C$ depending only on $\kappa$, $L$, $\Lambda$ and $T$ such that
\begin{align*}
	\|u^N\|_{\cL^{2,\infty}_{\sF}(s,T)}+\|\psi^N\|_{\cL^2_{\sF}(s,T;H^{1,2}(\bR^d))} &\leq C\left(\|N\wedge |F|\|_{\cL^2_{\sF}(0,T;L^2(\bR^d))} +\|N\wedge |G|\|_{L^2(\Omega,\sF_T;H^{1,2}(\bR^d))}\right) \\
  &\leq  C \left(\| F\|_{\cL^2_{\sF}(0,T;L^2(\bR^d))}
   +\|G\|_{L^2(\Omega,\sF_T;H^{1,2}(\bR^d))}\right).
\end{align*}
Letting $N\rightarrow \infty$, by Fatou's lemma and Jensen's inequality, we obtain
\begin{equation*}
\int_{\bR^d}\bigg(E[|G(y_T^{s,y})|] +\int_s^T E[|F_r(y_r^{s,y})|]\,dr\bigg)^2dy\\\leq C \left( \|G\|^2_{L^2(\Omega,\sF_T;H^{1,2}(\bR^d))} +\|F\|^2_{\cL^2_{\sF}(s,T;L^2(\bR^d))}\right).
\end{equation*}
This proves the desired estimates as well as the fact that $\Phi_{\cdot}(y_{\cdot}^{s,y})$ is a continuous and uniformly integrable semi-martingale for almost every $y\in\bR^d$. \hfill \qed
\end{appendix}

\end{document}